\documentclass[12pt]{amsart} 
\usepackage[mathscr]{eucal} 
\usepackage{amsmath,amsfonts}
\usepackage{datetime}
\parskip=\smallskipamount 
\newtheorem{theorem}{Theorem}[section] 
\newtheorem{proposition}[theorem]{Proposition} 
\newtheorem{corollary}[theorem]{Corollary} 
\newtheorem{lemma}[theorem]{Lemma} 
\theoremstyle{definition}

\newtheorem{examples}[theorem]{Examples}
 
\newtheorem{remark}[theorem]{Remark} 
\newtheorem{remarks}[theorem]{Remarks}
\makeatletter 
\@addtoreset{equation}{section} 
\makeatother

\newcommand{\CC}{{\mathbb C}} 
\newcommand{\NN}{{\mathbb N}}

\newcommand{\RR}{{\mathbb R}} 
 
\newcommand{\cA}{{\mathcal A}} 
\newcommand{\cB}{{\mathcal B}} 
\newcommand{\cC}{{\mathcal C}} 
\newcommand{\cD}{{\mathcal D}} 
\newcommand{\cE}{{\mathcal E}} 
\newcommand{\cF}{{\mathcal F}} 
\newcommand{\cG}{{\mathcal G}} 
\newcommand{\cH}{{\mathcal H}} 
\newcommand{\cI}{{\mathcal I}}
\newcommand{\cJ}{{\mathcal J}} 
\newcommand{\cK}{{\mathcal K}} 
\newcommand{\cL}{{\mathcal L}} 
\newcommand{\cM}{{\mathcal M}} 
 
\newcommand{\cP}{{\mathcal P}} 
\newcommand{\cR}{{\mathcal R}} 
\newcommand{\cS}{{\mathcal S}}

\newcommand{\fk}{\mathbf{k}}

\newcommand{\Ra}{\Rightarrow}

\newcommand{\ra}{\rightarrow} 
\newcommand{\ol}{\overline} 
\newcommand{\tl}{\widetilde}

\newcommand{\tr}{\operatorname{tr}} 
\let\phi=\varphi 
\newcommand{\iac}{\mathrm{i}}

\newcommand{\lin}{\operatorname{Lin}}

\newcommand{\nr}[1]{\vspace{0.1ex}\noindent\hspace*{12mm}\llap{\textup{(#1)}}} 
 
\begin{document} 
\title[Invariant Kernels with Values Continuously Adjointable Operators]{Representations 
of $*$-semigroups  associated to invariant kernels with values continuously adjointable 
operators}\thanks{The 
second named author's work supported by a grant of the Romanian 
National Authority for Scientific Research, CNCS  UEFISCDI, project number
PN-II-ID-PCE-2011-3-0119.}
 
 \date{\today,\ \currenttime}
%\time{\now}
 
 \author[S. Ay]{Serdar Ay}
 \address{Department of Mathematics, Bilkent University, 06800 Bilkent, Ankara, 
Turkey}
 \email{serdar@fen.bilkent.edu.tr}
 
\author[A. Gheondea]{Aurelian Gheondea} 
\address{Department of Mathematics, Bilkent University, 06800 Bilkent, Ankara, 
Turkey, \emph{and} Institutul de Matematic\u a al Academiei Rom\^ane, C.P.\ 
1-764, 014700 Bucure\c sti, Rom\^ania} 
\email{aurelian@fen.bilkent.edu.tr \textrm{and} A.Gheondea@imar.ro} 

\begin{abstract} We consider positive semidefinite 
kernels valued in the $*$-algebra of continuously
adjointable operators on a VH-space
(Vector Hilbert space in the sense of Loynes) and that are
invariant under actions of $*$-semigroups. For such a kernel we obtain two
necessary and sufficient boundedness conditions in order for there to exist
$*$-representations of the underlying $*$-semigroup on a
VH-space linearisation, equivalently, on a reproducing kernel
VH-space. We exhibit several situations when the latter boundedness condition is automatically 
fulfilled. For example, when specialising to the case of Hilbert modules over locally
$C^*$-algebras, we show that both boundedness conditions are automatically
fulfilled and, consequently, this general approach provides a rather direct
proof of the general Stinespring-Kasparov type dilation theorem for
completely
positive maps on locally $C^*$-algebras and with values adjointable operators on Hilbert
modules over locally $C^*$-algebras.
\end{abstract} 

\maketitle 

\section*{Introduction}

In 1965, R.M.~Loynes published two articles \cite{Loynes1} and \cite{Loynes2}
where he considered generalisations of the notions of inner product space and
of Hilbert space, that he called VE-space (Vector Euclidean space)
and, respectively, VH-space (Vector Hilbert space). These are vector spaces
on which there are ``inner products'' with values in certain ordered
$*$-spaces, hence ``vector valued inner products'', see subsections
\ref{ss:vestlo}--\ref{ss:vhs} for precise definitions. 
His motivation was coming from stochastic processes \cite{Loynes3}
and the main results refer to a generalisation of B.~Sz.-Nagy' Dilation Theorem
\cite{BSzNagy} for operator valued positive 
semidefinite maps on
$*$-semigroups \cite{Loynes1}, and to some other results on spectral theory of linear bounded
operators on VH-spaces \cite{Loynes2}. These ideas have been followed in prediction theory
\cite{Chobanyan},
\cite{Weron}, \cite{WeronChobanyan}, in dilation theory \cite{GasparGaspar},
\cite{GasparGaspar1}, \cite{GasparGaspar2}, and a few others.

On the other hand, special cases of VH-spaces have been later
considered independently of the Loynes' articles. Thus,
the concept of Hilbert module over a $C^*$-algebra was introduced
in 1973 by W.L. Paschke in \cite{Paschke}, following
I.~Kaplansky \cite{Kaplansky}, and independently by M.A. Rieffel one year later in
\cite{Rieffel}, and these two articles triggered a whole domain
of research, see e.g.\ \cite{Lance} and \cite{ManuilovTroitsky} 
and the rich bibliography cited there. Hilbert
modules over $C^*$-algebras are special cases of VH-spaces. 
Dilation theory plays a very important role in this theory and there are many
dilation results of an impressive diversity, but
the domain of Hilbert modules over $C^*$-algebras remained unrelated to that
of VH-spaces.  Another special case of a VH-space is that of Hilbert modules
over $H^*$-algebras of P.P.~Saworotnow \cite{Saworotnow}. 
Also, in 1985 A.~Mallios \cite{Mallios} and later in
1988 N.C.~Phillips \cite{Phillips} introduced and
studied the concept of Hilbert module over locally $C^*$-algebra, which is
yet another particular case of VH-space over an admissible space. The theory
of Hilbert spaces over locally $C^*$-algebras is an active domain of research
as well, e.g.\ see \cite{Joita2} and the rich bibliography cited there.

Taking into account the importance and the diversity of dilation theory,
e.g.\ see \cite{Arveson1}, 
it is natural to ask for its unification under a general
framework. Historically, the theory of 
positive semidefinite kernels, having values in operator spaces or $*$-ordered
spaces, e.g.\ see 
\cite{EvansLewis}, \cite{ParthasarathySchmidt}, \cite{Murphy}, \cite{Heo}, and
\cite{Szafraniec}, 
to cite just a few, turned out to provide, to a certain extent, 
such a unification framework, that can be made much more efficient when 
a certain "symmetry" is added, more 
precisely, the invariance under the action of a $*$-semigroup, 
e.g.\ see \cite{ConstantinescuGheondea2}. Following \cite{Gheondea}, in this
article we show that this unifying framework becomes significantly more successful 
when kernels with values linear operators on VH-spaces are employed. 
In \cite{Gheondea} there is one extra assumption on the range of
the kernels, namely that of boundedness in the sense of Loynes, which
restricts the area of applicability to $C^*$-algebras 
and, in order to unify other dilation results, 
e.g.\ the dilation of completely positive maps on Hilbert modules over locally
$C^*$-algebras, see \cite{Joita}, the boundedness condition should be relaxed. 

This article is one step further in the programme, initiated in
\cite{Gheondea}, of unifying dilation results
under a setting comprising positive semidefinite kernels that are invariant 
under actions of $*$-semigroups and with values continuous and continuously
adjointable operators on VH-spaces, and 
a continuation of the work \cite{AyGheondea} in which we obtained a nontopological 
version of this kind of dilation theorem. From this point of view, the
main result of this article is Theorem~\ref{t:vhinvkolmo2} that provides 
two necessary and sufficient conditions for the existence of 
$*$-representions of the given
$*$-semigroup by continuous and continuously adjointable operators on
VH-spaces. The boundedness condition (b1) in Theorem~\ref{t:vhinvkolmo2} 
is the analog of the
celebrated Sz.-Nagy's boundedness condition \cite{BSzNagy} (see
\cite{Szafraniec} for a historical perspective of this issue) and is related to the
continuity of linear operators in the range of the $*$-representation, an
obstruction caused by the gap between $*$-semigroups and groups. The boundedness
condition (b2) from Theorem~\ref{t:vhinvkolmo2} is new and refers to an
obstruction related to the
continuity of the adjoint operators which, in the case of VH-spaces, is not 
automatic.

Theorem~\ref{t:vhinvkolmo2} unifies most of the known dilation theorems for
operator valued maps, in the
chain of the two classical Naimark's theorems for operator valued positive 
semidefinite functions on commutative groups \cite{Naimark2} and, respectively, 
for semispectral measures \cite{Naimark1}, that is, 
the Stinespring's Theorem \cite{Stinespring} for operator valued completely 
positive maps on
$C^*$-algebras, the Sz.-Nagy's Theorem \cite{BSzNagy} for operator valued positive
semidefinite functions on $*$-semigroups and its VH-space
generalization of Loynes \cite{Loynes1}, as well as 
the dilation theorems for completely positive maps on $C^*$-algebras with
values adjointable operators on Hilbert modules over $C^*$-algebras of
Kasparov \cite{Kasparov} and that 
for completely positive maps on locally $C^*$-algebras with values adjointable
operators on Hilbert modules over locally $C^*$-algebras of M.~Joi\c ta
\cite{Joita}. In this article, we explicitly show how the 
latter is obtained as a consequence of Theorem~\ref{t:vhinvkolmo3}.

In the following we briefly describe the contents of this article. The first
section is dedicated to  
notation and preliminary results on VH-spaces and their linear
operators. Since we built on the  
fabric of dilation theory on VE-spaces over ordered $*$-spaces, we first 
briefly review necessary concepts, results, and constructions from
\cite{AyGheondea}. One of the main
mathematical objects used in this research is that of Loynes' admissible space,
that is, a complete 
topologically ordered $*$-space. A list of nine examples, that we carefully
present, indicates the unifying  
potential of this concept. VH-spaces and their linear operators are discussed in 
Subsection~\ref{ss:vhs}. Here, we draw attention to Lemma~\ref{l:topology} 
that clarifies the locally convex topology on VH-spaces and to the six 
generic examples that illustrate the  
unifying potential of the concept of VH-space. Three of the main technical
obstructions in this theory  
are related to the lack of a general Schwarz type inequality, to the existence
of non-orthocomplemented VH-subspaces, and to the lack
of a reliable substitute 
for the Riesz's Representation Theorem for continuous linear functionals. 
Consequently, many 
technical ingredients that are used in this article gravitates around finding
sufficiently powerful  
surrogates of these missing tools. In this respect, in
Corollary~\ref{c:4schwarz} and Lemma~\ref{l:redunt} 
we obtain some surrogates of the Schwarz inequality and then
refinements are performed in Subsection~\ref{ss:cb2t}.

The main section of this article refers to positive semidefinite kernels with
values continuous and 
continuously adjointable linear operators on VH-spaces. 
Again, since we built on dilation results on VE-spaces over 
$*$-ordered spaces investigated in \cite{AyGheondea}, we first review 
the necessary terminology, results, and constructions corresponding to
positive semidefinite kernels in the nontopological case. 
When working with kernels, there is a 
paradigmatic idea that the natural approach is through reproducing kernel 
spaces, 
e.g.\ see \cite{Aronszajn}, \cite{Szafraniec} and the rich bibliography cited
there. For this reason, we first investigate basic properties of VH-space
linearisations (Kolmogorov decompositions) and their interplay with
reproducing kernel VH-spaces which, at this 
level of generality, require a careful
treatment: most of the properties that we expect are true, but some of the
proofs are rather different. We stress
that our approach of the dilation constructions is through reproducing kernel
spaces that has  
substantial advantages: the objects that are built preserve their concrete
character to the largest  
possible extent, for example we always obtain (operator valued) function spaces and not
abstract quotient  
spaces, in contrast to the GNS construction which is traditionally extensively used in dilation theory.

Theorem~\ref{t:vhinvkolmo2}, the main result of this article, emphasises 
the boundedness condition (b1), the analog of the Sz.-Nagy's boundedness condition, 
and the boundedness condition (b2) that shows up due to topological obstructions
of dealing with linear operators on locally convex spaces, especially in
connection with the  
topological pathologies related to multiplication. Recently, a related
phenomenon has been discussed by W.~\.Zelazko \cite{Zelazko} who introduced a
class of continuous linear operators on locally convex spaces $\cE$ 
for which there is a certain control of the growth of their powers uniformly
on $\cE$, that he called m-topologisable (multiplicatively topologisable), 
see also \cite{Bonet}. In Subsection~\ref{ss:cb2t} we show that, when a positive semidefinite
kernel has m-topologisable operators on its whole diagonal, a stronger condition than the
boundedness condition (b2) is obtained by an iteration method, previously
employed in spectral theory \cite{Krein}, \cite{Reid}, \cite{Lax},
\cite{Dieudonne}, in particular, m-topologisability propagates throughout the kernel. 
However, the question whether condition (b2) holds
at the level of generality of positive semidefinite kernels with values
continuous and continuously adjointable operators on VH-spaces remains open.

The last section is dedicated to show the unifying coverage of our
Theorem~\ref{t:vhinvkolmo2}, by providing a direct proof of the dilation
theorem from \cite{Joita}. Theorem~\ref{t:vhinvkolmo3} is a
remarkable consequence of Theorem~\ref{t:vhinvkolmo2} and of the previously
obtained results for m-topologisable operators, which turns out to be the case
in this context. This theorem shows that, for invariant positive
semidefinite kernels with values adjointable operators on a Hilbert module over
a locally $C^*$-algebra, the boundedness condition (b2) is automatic, hence the
existence of $*$-representations on a Hilbert locally $C^*$-module
linearisation of the kernel, equivalently, on the
reproducing kernel Hilbert locally $C^*$-module of the kernel, depends only on
the boundedness condition (b1).
Finally, we point out why the boundedness
conditions (b1) discussed above is automatic as well in the special case of
completely positive maps on locally $C^*$-algebras and with values adjointable
operators on 
Hilbert modules over locally $C^*$-algebras, by an adaptation of the technique
of Murphy \cite{Murphy} that solves the nonunital case by approximate
identities in locally $C^*$-algebras.

\section{Preliminaries}

In this section we briefly review most of the definitions and some basic
facts on ordered $*$-spaces, VE-spaces over ordered $*$-spaces, and their
linear operators, then review and get some facts on VH-spaces over admissible
spaces and their linear operators.

\subsection{VE-Spaces and Their Linear Operators.}\label{ss:vestlo}
A complex vector space $Z$ is called \emph{ordered $*$-space}, see
\cite{PaulsenTomforde}, if:
\begin{itemize}
\item[(a1)] $Z$ has an \emph{involution} $*$, that is, a map 
$Z\ni z\mapsto z^*\in Z$ 
that is \emph{conjugate linear} 
(($s x+t y)^*=\ol s x^*+\ol t y^*$ for all 
$s,t\in\CC$ and all $x,y\in Z$) and \emph{involutive} 
($(z^*)^*=z$ for all $z\in Z$). 
\item[(a2)] In $Z$ there is a \emph{cone} $Z^+$ ($s x+t y\in Z^+$ 
for all numbers $s,t\geq 0$ and all $x,y\in Z^+$), that is  
\emph{strict} ($Z^+\cap -Z^+=\{0\}$), and consisting of 
\emph{selfadjoint elements} 
only  ($z^*=z$ for all 
$z\in Z^+$). This cone is used to define a \emph{partial order} on the real 
vector space of all selfadjoint elements in $Z$: 
$z_1  \geq z_2$ if $z_1-z_2 \in Z^+$.
\end{itemize}

Recall that a \emph{$*$-algebra} $\cA$ is a complex algebra onto which 
there is defined an \emph{involution} $\cA\ni a\mapsto a^*\in\cA$, that is, 
$(\lambda a+\mu b)^*=\overline \lambda a^*+\overline \mu b^*$, 
$(ab)^*=b^*a^*$, and $(a^*)^*=a$, for all $a,b\in\cA$ and all 
$\lambda,\mu\in\CC$. 

An \emph{ordered $*$-algebra} $\cA$ is a $*$-algebra 
such that it is an ordered $*$-space, more precisely, it has the following
property.
\begin{itemize}
\item[(osa1)] There exists a strict cone $\cA^+$ in $\cA$ such that for any 
$a\in\cA^+$ we have $a=a^*$.
\end{itemize}
Clearly, any ordered $*$-algebra is an ordered $*$-space. In particular, 
given $a\in\cA$, we denote $a\geq 0$ if $a\in\cA^+$ and, for 
$a=a^*\in\cA$ and $b=b^*\in\cA$, we denote $a\geq b$ if $a-b\geq 0$.

Given a complex linear space $\cE$ and an
ordered $*$-space space $Z$, a \emph{$Z$-gramian}, also called a 
\emph{$Z$-valued inner product}, is, by definition, a mapping  
$\cE\times \cE\ni (x,y) \mapsto [x,y]\in Z$ subject to 
the following properties:
\begin{itemize}
\item[(ve1)] $[x,x] \geq 0$ for all $x\in \cE$, and $[x,x]=0$ if and only if 
$x=0$.
\item[(ve2)] $[x,y]=[y,x]^*$ for all $x,y\in\cE$.
\item[(ve3)] $[x,\alpha y_1+\beta y_2]=\alpha 
[x,y_1]+\beta [x,y_2]$ for all $\alpha,\beta\in \mathbb{C}$ and 
all $x_1,x_2\in \cE$.
\end{itemize}

A complex linear space $\cE$ onto which a $Z$-gramian 
$[\cdot,\cdot]$ is specified, for a 
certain ordered $*$-space $Z$, is called a \emph{VE-space} 
(Vector Euclidean space) over $Z$, cf.\ \cite{Loynes1}. 
 
Given a pairing 
$[\cdot,\cdot]\colon \cE\times \cE\ra Z$, where $\cE$ is some vector space
and $Z$ is an ordered $*$-space, 
and assuming that $[\cdot,\cdot]$ satisfies only the
axioms (ve2) and (ve3), then a
\emph{polarisation formula} holds
\begin{equation}\label{e:polar} 4[x,y]=\sum_{k=0}^3 \iac^k 
[x+\iac^k y,x+\iac^k y],\quad x,y\in \cE.
\end{equation} In particular, this formula holds on a VE-space and it
shows that the $Z$-gramian is 
perfectly defined by the $Z$-valued quadratic map 
$\cE\ni x\mapsto [x,x]\in Z$.

A \emph{VE-spaces isomorphism} is, by definition,
a linear bijection $U\colon \cE\ra \cF$, for two VE-spaces 
over the same ordered $*$-space $Z$, which is \emph{isometric}, that is, 
$[Ux,Uy]_\cF=[x,y]_\cE$ for all $x,y\in \cE$.

In general VE-spaces, an analog of the Schwarz Inequality may not hold but 
some of its consequences can be proven using slightly different 
techniques, cf.\ \cite{Loynes1}, \cite{Loynes2}. 
Given two VE-spaces $\cE$ and $\cF$, over the same ordered $*$-space 
$Z$, one can consider the vector space $\cL(\cE,\cF)$ of all
linear operators $T\colon \cE\ra\cF$. 
A linear operator $T\in\cL(\cE,\cF)$ is 
called \emph{adjointable} if there exists $T^*\in\cL(\cF,\cE)$ such that
\begin{equation}\label{e:adj} [Te,f]_\cF=[e,T^*f]_\cE,\quad e\in\cE,\ f\in\cF.
\end{equation} The operator $T^*$, if it exists, is uniquely determined by $T$ 
and called its \emph{adjoint}.
Since an analog of the Riesz Representation Theorem for VE-spaces
may not exist, in general, 
there may be not so many adjointable operators. Denote by 
$\cL^*(\cE,\cF)$ the vector space of all adjointable operators from 
$\cL(\cE,\cF)$.
Note that $\cL^*(\cE)=\cL^*(\cE,\cE)$ 
is a $*$-algebra with respect to the involution $*$ 
determined by the operation of taking the adjoint. 

An operator $A\in\cL(\cE)$ is called \emph{selfadjoint} if
$ [Ae,f]=[e,Af]$, for all $ e,f\in \cE$.
Any selfadjoint operator $A$ is adjointable and 
$A=A^*$.
By the polarisation formula \eqref{e:polar}, $A$ is selfadjoint if and only if
$ [Ae,e]=[e,Ae]$, for all $e\in\cE$.
An operator $A\in\cL(\cE)$ is \emph{positive} if
$[Ae,e]\geq 0$, for all $e\in\cE$.
Since the cone $Z^+$ consists of selfadjoint elements only, 
any positive operator is selfadjoint and hence adjointable. 
Note that any VE-space isomorphism $U$
is adjointable, invertible, and $U^*=U^{-1}$, hence, equivalently, 
we can call it \emph{unitary}.

%%%%%%%%% VE-module %%%%%%%%%%%%%
A \emph{VE-module $\cE$ over an ordered $*$-algebra $\cA$} is a right 
$\cA$-module on which there exists an \emph{$\cA$-gramian} 
$[\cdot,\cdot]_\cE\colon \cE\times \cE \ra \cA$ with respect to which it
is a VE-space, that is, (ve1)-(ve3) hold, and, in addition,
\begin{itemize}
\item[(vem)] $[e,fa+gb]_\cE=[e,f]_\cE a+[e,g]_\cE b$ for all $e,f,g\in\cE$ and 
all $a,b\in\cA$.
\end{itemize}

Given an ordered $*$-algebra $\cA$ and two VE-modules $\cE$ and $\cF$
over $\cA$, an operator $T\in\cL(\cE,\cF)$ is called a \emph{module map} if
\begin{equation*} T(ea)=T(e)a,\quad e\in\cE,\ a\in\cA.
\end{equation*}
It is easy to see that
any operator $T\in\cL^*(\cE,\cF)$ is a module map, e.g.\ see \cite{AyGheondea}.

\subsection{Admissible Spaces.}\label{ss:as}
The complex vector space $Z$ is called \emph{topologically ordered 
$*$-space} if it is an ordered $*$-space, that is, axioms (a1) and (a2)
hold and, in addition,
\begin{itemize}
\item[(a3)] $Z$ is a \emph{Hausdorff locally convex space}.
\item[(a4)] The topology of $Z$ is \emph{compatible} with the partial 
ordering in the 
sense that there exists a base of the topology, linearly generated by a 
family of neighbourhoods $\{ C \}_{C\in\cC_0}$ of the origin that
are absolutely convex and 
\emph{solid}, in the sense that, 
if $x\in C$ and $y\in Z$ are such that $0\leq y\leq x$, then $y\in C$.
\end{itemize}

\begin{remark}\label{r:increasing}
Axiom (a4) is equivalent with the following one:
\begin{itemize}
\item[(a4$^\prime$)] There exists a collection of seminorms 
$\{p_j\}_{j\in \cJ}$ defining the 
topology of $Z$ that, for any $j\in\cJ$, $p_j$ is \emph{increasing}, in the 
sense that, $0\leq x\leq y$ implies $p_j(x)\leq p_j(y)$. 
\end{itemize}

To see this, e.g.\ see Lemma 1.1.1 and Remark 1.1.2 of 
\cite{Ciurdariu}, letting $\cC_0$ be a family of 
open, absolutely convex and solid neighbourhoods of the origin 
defining the topology of $Z$, 
for each $C\in\cC_0$, consider the Minkowski seminorm $p_C$ associated 
to $C$, 
\begin{equation}\label{e:seminorm}
p_C(x)=\mbox{inf}\{ \lambda \mid \lambda>0,\quad x\in\lambda C \},\quad x\in Z.
\end{equation} 
Clearly, $\{p_C \mid C\in\cC_0\}$ define the topology of 
$Z$. Moreover, $p_C$ is increasing. To see this, for any $\epsilon> 0$, there exists 
$p_C(x)\leq \lambda_{\epsilon}\leq p_C(x)+\epsilon$ 
such that $x\in\lambda_{\epsilon}C$. Since 
$C$ is balanced, $\lambda_{\epsilon}C \subset (p_C(x)+\epsilon)C$, 
so $x\in(p_C(x)+\epsilon)C$. As $C$ is also solid, if $0\leq y\leq x$, then we have 
$y\in (p_C(x)+\epsilon)C$, from which we obtain $p_C(y)\leq p_C(x)+\epsilon$. 
Since $\epsilon > 0$ was arbitrary, we have that $p_C(y)\leq p_C(x)$. 

Conversely, given any increasing continuous 
seminorm $p$ on $Z$, the set 
\begin{equation*}
C_p:=\{ x\in Z \mid p(x)<1 \}
\end{equation*}  
is absolutely convex. Moreover, it is solid 
since, if $x\in C_p$ with $0\leq y\leq x$, 
then $p(y)\leq p(x) < 1$, so $y\in C_p$.  \end{remark}

Given a family $\cC_0$ of absolutely convex and solid neighbourhoods of the
origin that generates the topology of $Z$, we denote by
$S_{\cC_0}(Z)=\{p_C\mid C\in\cC_0\}$, where $p_C$ is the Minkowski seminorm
associated to $C$ as in \eqref{e:seminorm}. 
The collection of all continuous 
increasing seminorms on $Z$ is denoted by $S(Z)$. 
As a consequence of Remark~\ref{r:increasing}, 
$S(Z)$ is in bijective correspondence 
with the family $\cC$ of all open, absolutely convex
and solid neighbourhoods of the origin. 
Note that $S(Z)$ is a directed set: given 
$p,q\in S(Z)$, consider $r:=p+q$. In fact, 
$S(Z)$ is a cone, i.e.\ it is closed under all 
finite linear combinations with positive coefficients.  

$Z$ is called an \emph{admissible space}, cf.\ \cite{Loynes1}, if, in addition 
to the axioms (a1)--(a4),
\begin{itemize}
\item[(a5)] The cone $Z_+$ is \emph{closed}, with respect to the
specified topology of $Z$.
\item[(a6)] The topology on $Z$ is complete.
\end{itemize}

Finally, if, in addition to the axioms (a1)--(a6), the space $Z$ satisfies
also the following axiom:
\begin{itemize}
\item[(a7)] With respect to the specified partial ordering,  any bounded 
monotone 
sequence is convergent.
\end{itemize}
then $Z$ is called a \emph{strongly admissible} space \cite{Loynes1}.

\begin{examples}\label{ex:as}
 (1) Any $C^*$-algebra $\cA$ is an admissible space, as well 
as any closed $*$-subspace $\cS$ of a $C^*$-algebra $\cA$, with the 
positive cone $\cS^+=\cA^+\cap\cS$ and all other operations (addition, 
multiplication with scalars, and involution) inherited from $\cA$.

(2) Any pre-$C^*$-algebra is a topologically ordered $*$-space. Any 
$*$-subspace $\cS$ of a pre-$C^*$-algebra $\cA$ is a topologically ordered 
$*$-space, with the positive cone $\cS^+=\cA^+\cap\cS$ and all other 
operations inherited from $\cA$.

(3) Any locally $C^*$-algebra, cf.\ \cite{Inoue}, \cite{Phillips}, 
(definition is recalled in Subsection \ref{ss:hmolcsa}) is an admissible 
space. 
In particular, any closed $*$-subspace $\cS$ of a locally $C^*$-algebra $\cA$, 
with the cone $\cS_+=\cA^+\cap\cS$ and all other operations inherited from 
$\cA$, is an admissible space. 

(4) Any locally pre-$C^*$-algebra is a topologically ordered $*$-space. Any
$*$-subspace $\cS$ of a locally pre-$C^*$-algebra is a topologically ordered
$*$-space, with $\cS^+=\cA^+\cap\cS$ and all other operations inherited from 
$\cA$.

(5) Let $\cH$ be an infinite dimensional separable Hilbert space and let
$\cC_1$ be 
the trace-class ideal, that is, the collection of all linear bounded operators
$A$ 
on $\cH$ such that $\tr(|A|)<\infty$. $\cC_1$ is a $*$-ideal of $\cB(\cH)$ and
complete under the 
norm $\|A\|_1=\tr(|A|)$. Positive elements in $\cC_1$ are defined in the sense 
of positivity in $\cB(\cH)$. In addition, the norm $\|\cdot\|_1$ is
increasing, since 
$0\leq A\leq B$ implies $\tr(A)\leq \tr(B)$, hence $\cC_1$ is a normed
admissible space. 

(6) Let $V$ be a complex Banach space and let
$V^\prime$ be its conjugate dual space. On the vector space 
$\mathcal{B}(V,V^\prime)$ of all bounded linear
operators $T\colon V\rightarrow V^\prime$, a natural notion of positive 
operator can be defined: $T$ is \emph{positive} if $(Tv)(v)\geq 0$ for all 
$v\in V$. Let
$\mathcal{B}(V,V^\prime)^+$ be the collection of all positive operators  and 
note that it is a strict cone that is closed with respect to the weak operator 
topology. The involution $*$ 
in $\mathcal{B}(V,V^\prime)$ is defined in the following way: for any
$T\in\mathcal{B}(V,V^\prime)$, $T^*=T^\prime|V$, that is, the restriction to 
$V$ of the dual operator 
$T^\prime\colon V^{\prime\prime}\rightarrow V^\prime$. With respect to 
the weak operator topology, the cone $\mathcal{B}(V,V^\prime)^+$, and the 
involution $*$ just defined, $\mathcal{B}(V,V^\prime)$ becomes an admissible 
space. 
See A.~Weron \cite{Weron}, as well as D.~Ga\c spar and P.~Ga\c spar 
\cite{GasparGaspar}.

(7) Let $X$ be a nonempty set and denote by $\cK(X)$ the collection of 
all complex valued kernels on $X$, that is, $\cK(X)=\{k\mid k\colon X\times 
X\rightarrow \CC\}$, considered as a complex vector space with the 
operations of addition and multiplication of scalars defined elementwise. 
An involution 
$*$ can be defined on $\cK(X)$ as follows: $k^*(x,y)=\overline{k(y,x)}$, 
for all 
$x,y\in X$ and all $k\in\cK(X)$. The cone $\cK(X)^+$ consists of all 
\emph{positive semidefinite} kernels, that is, those kernels $k\in\cK(X)$ with 
the property that, for any $n\in\NN$ and any $x_1,\ldots,x_n\in X$, the 
complex matrix $[k(x_i,x_j)]_{i,j=1}^n$ is positive semidefinite. Then $\cK(X)$
is an ordered $*$-space.

Further, consider the set $\cP_0(X)$ of all finite subsets of $X$. For each 
$A\in\cP_0(X)$, let $A=\{x_1,\ldots,x_n\}$ and define the seminorm 
$p_A\colon\cK(X)\ra \RR$ by
\begin{equation*} p_A(k)=\|[k(x_i,x_j)]_{i,j=1}^n\|,\quad k\in\cK(X),
\end{equation*} the norm being the operator norm of the $n\times n$ matrix 
$[k(x_i,x_j)]_{i,j=1}^n$. Since a reordering of the elements $x_1,\ldots,x_n$ 
produces a unitary equivalent matrix, the definition of $p_A$ does not depend 
on which order of the elements of the set $A$ is considered.
It is easy to see that each seminorm $p_A$ 
is increasing and that, with the locally convex topology defined by 
$\{p_A\}_{A\in\cP_0(X)}$, $\cK(X)$ is an admissible space.

(8) Let $\cA$ and $\cB$ be two $C^*$-algebras. Recall that, in this case, 
the specified strict cone $\cA^+$ linearly generates $\cA$. On 
$\cL(\cA,\cB)$, the vector space of all linear maps $\phi\colon \cA\ra\cB$, 
we define an involution: $\phi^*(a)=\phi(a^*)^*$, for all $a\in\cA$. A linear 
map $\phi\in\cL(\cA,\cB)$ is called positive if $\phi(\cA^+)\subseteq \cB^+$. It
is easy to see that $\cL(\cA,\cB)^+$, the collection of all positive maps from 
$\cL(\cA,\cB)$, is a cone, and that it is strict because $\cA^+$ 
linearly generates $\cA$. In addition, any $\phi\in\cL(\cA,\cB)^+$ 
is selfadjoint, 
again due to the fact that $\cA^+$ linearly generates $\cA$. Consequently,
$\cL(\cA,\cB)$ has a natural structure of ordered $*$-space.

On $\cL(\cA,\cB)$ we consider the collection of seminorms 
$\{p_a\}_{a\in\cA^+}$ defined by $p_a(\phi)=\|\phi(a)\|$, for all 
$\phi\in\cL(\cA,\cB)$. All these seminorms are increasing and the topology
generated by $\{p_a\}_{a\in\cA^+}$ is Hausdorff and complete. Consequently,
$\cL(\cA,\cB)$ is an admissible space.

With a slightly more involved topology, it can be shown that the same
conclusion holds for the case when $\cA$ and $\cB$ are locally $C^*$-algebras.

(9) Let $\{Z_\alpha\}_{\alpha\in A}$ be a family of admissible spaces such that,
for each $\alpha\in A$, $Z_\alpha^+$ is the specified strict cone of positive
elements in $Z_\alpha$, and the topology of $Z_\alpha$ is generated by the
family of increasing seminorms $\{p_{\alpha,j}\}_{j\in\cJ_\alpha}$. 
On the product space $Z=\prod_{\alpha\in A}Z_\alpha$ let 
$Z^+=\prod_{\alpha\in A}Z_\alpha^+$ and observe that $Z^+$ is a
strict cone. Letting the involution $*$ on $Z$ be defined elementwise, it
follows that $Z^+$ consists on selfadjoint elements only. In this way, $Z$
is an ordered $*$-space.

For each $\beta\in A$ and each $j\in\cJ_\beta$, let
\begin{equation}\label{e:qs}% q seminorms
q_j^{(\beta)}((z_\alpha)_{\alpha \in A})=p_j^{(\beta)}(z_\beta),\quad
(z_\alpha)_{\alpha\in A}\in Z.
\end{equation} It is easy to show that $q_j^{(\beta)}$ is an increasing
seminorm on $Z$ and that, with the topology generated by the family of
increasing seminorms $\{q_j^{(\beta)}\}_{\substack{\beta\in A\\ j\in\cJ_\beta}}$, 
$Z$ becomes an admissible space.
\end{examples}

\subsection{Vector Hilbert Spaces and Their Linear Operators.}\label{ss:vhs}
If $Z$ is a topologically ordered $*$-space, 
any VE-space $\cE$ over $Z$ can be made in a natural way
into a Hausdorff locally convex space by considering the topology $\tau_\cE$, 
the weakest topology on $\cE$ that makes the quadratic map
$Q: E\ni h\mapsto [h,h]\in Z$ continuous. More 
precisely, letting $\cC_0$ be a collection of open,  
absolutely convex and solid 
neighbourhoods of the origin in $Z$, that generates the topology 
of $Z$ as in axiom (a5), the collection of sets
\begin{equation}\label{e:ujex}D_C=\{x\in \cE\mid [x,x]\in C\},\quad C\in\cC_0,
\end{equation} is a topological base of open and absolutely convex 
neighbourhoods of the origin of $\cE$ that linearly generates $\tau_\cE$,
cf.\ \cite{Loynes1}. 
We are interested in explicitly 
defining the topology $\tau_\cE$ in terms of seminorms. 

\begin{lemma}\label{l:topology} Let $Z$ be a topologically ordered $*$-space
and $\cE$ a VE-space over $Z$.
\begin{itemize}
\item[(1)] $(\cE;\tau_\cE)$ is a Hausdorf locally convex space.
\item[(2)] For every continuous increasing seminorm $p$ on $Z$
\begin{equation}\label{e:qujeh}\tl{p}(h)=p([h,h])^{1/2},\quad h\in\cE,
\end{equation} is a continuous seminorm on $(\cE;\tau_\cE)$.
\item[(3)] Let $\{p_j\}_{j\in\cJ}$ be a family of increasing 
seminorms defining the topology of $Z$ 
as in axiom (a4$^\prime$). Then, with the definition \eqref{e:qujeh}, 
the family of seminorms $\{\tl{p}_j\}_{j\in\cJ}$ generates
$\tau_\cE$.
\item[(4)] The gramian $[\cdot,\cdot]\colon \cE\times \cE\ra Z$ 
is jointly continuous.
\end{itemize}
\end{lemma}

Statements (1) and (4) are proven in Theorem~1 in \cite{Loynes1}. 
Statement (2) is 
claimed in Proposition 1.1.1 in \cite{Ciurdariu} but, unfortunately, the 
proof provided there is irremediably flawed, so we provide full details.

\begin{proof}[Proof of Lemma~\ref{l:topology}] We first prove that, if $p$
is a continuous and increasing seminorm on $Z$, $\tl{p}$ is a quasi 
seminorm on $\cE$. Indeed, for any $\lambda\in\CC$ and any $h\in\cE$
\begin{equation*} \tl{p}(\lambda h)=p([\lambda h,\lambda h])^{1/2}= |\lambda| 
p([h,h])^{1/2}=|\lambda|\tl{p}(h),
\end{equation*} hence $\tl{p}$ is positively homogeneous.

For arbitrary $h,k\in\cE$ we have
\begin{equation*}[h\pm k,h\pm k]=[h,h]+[k,k]\pm[h,k]\pm[k,h]\geq 0,
\end{equation*} in particular,
\begin{equation}\label{e:haska} [h,k]+[k,h]\leq [h,h]+[k,k].
\end{equation}
and
\begin{equation}\label{e:heka}
0\leq [h\pm k,h\pm k]\leq [h-k,h-k]+[h+k,h+k]=2([h,h]+[k,k]).
\end{equation}
Since $p$ is increasing, it follows that
\begin{align*} \tl{p}(h+k) &=\bigl(p([h+k,h+k])\bigr)^{1/2}\leq 
\sqrt{2}(p([h,h])+p([k,k])^{1/2}\\
&\leq \sqrt{2}\bigl(p([h,h])^{1/2}+p([k,k])^{1/2}\bigr)=
\sqrt{2}\bigl(\tl{p}(h)+\tl{p}(k)\bigr).
\end{align*}
This concludes the proof that $\tl{p}$ is a quasi seminorm.

Also, since $\tl{p}$ is the composition of the square root function 
$\sqrt{\phantom{t}}$, a homeomorphism of $\RR_+$ onto itself, 
with $p$ and the quadratic map $\cE\ni x\mapsto [x,x]\in Z$, clearly $\tl{p}$ is
continuous with respect to the topology $\tau_\cE$. This observation shows 
that, if $\{p_j\}_{j\in\cJ}$ is a family of increasing seminorms generating 
the topology of $Z$, then $\{\tl{p}_j\}_{j\in\cJ}$ is a family of quasi seminorms 
generating $\tau_\cE$. In particular, $(\cE;\tau_\cE)$ is a 
topological vector space.

We prove now that $\tl{p}$ satisfies the triangle inequality, hence it is a
seminorm. To see this, consider the unit quasi ball 
\begin{equation*} U_{\tl{p}}=\{h\in \cE\mid \tl{p}(h)<1\}.
\end{equation*} Since $\tl{p}$ is continuous, $U_{\tl{p}}$ is open, 
hence absorbing for each of its points. Since $\tl{p}$ is 
positively homogeneous, 
$U_{\tl{p}}$ is balanced. We prove that $U_{\tl{p}}$ is convex as well. 
Let $h,k\in U_{\tl{p}}$ and $0\leq t\leq 1$ arbitrary.
Then,
\begin{align*} 0 \leq [th+(1-t)k,th+(1-t)k] & =t^2[h,h]+(1-t)^2[k,k]
+t(1-t)\bigl([h,k]+[k,h]\bigr)\\ 
\intertext{and then using \eqref{e:haska},}
& \leq t^2[h,h]+(1-t)^2[k,k]+t(1-t)
\bigl([h,h]+[k,k]\bigr)\\ & = t[h,h]+(1-t)[k,k],
\end{align*} hence, since $p$ is increasing, it follows
\begin{equation*}\tl{p}(th+(1-t)k)=p\bigl([th+(1-t)k,th+(1-t)k]\bigr)^{1/2}\leq\bigl( t p([h,h])+(1-t)p([k,k])\bigr)^{1/2}<1,
\end{equation*} hence $th+(1-t)k\in U_{\tl{p}}$.

It is a routine exercise to show that $\tl{p}$ is the gauge of $U_{\tl{p}}$
\begin{equation*} \tl{p}(h)=\inf\{t>0\mid h\in tU_{\tl{p}}\},
\end{equation*} hence, by Proposition~IV.1.14 in \cite{Conway}, 
it follows that $\tl{p}$ is a seminorm.

Statement (4) is a consequence of the polarisation formula \eqref{e:polar}.
\end{proof}

From now on, any time we have a VE-space $\cE$ over a topologically 
ordered $*$-space $Z$, 
we consider on $\cE$ the topology $\tau_\cE$ defined as in 
Lemma~\ref{l:topology}. With respect to this topology, we call $\cE$ a 
\emph{topological VE-space} over $Z$. Denote 
\begin{equation}\label{e:tves}% topology of VE-space
S(\cE):=S_{\cC}(\cE)=\{ \tl{p}_C \mid C\in\cC \},\end{equation} 
where $\cC$ is the collection of all 
open, absolutely convex and solid neighbourhoods of the origin of $Z$ as in 
\eqref{e:ujex}. 
Note that $S(\cE)$ is directed, more precisely,
given $\tl{p}_C,\tl{p}_D\in S(\cE)$ consider 
$S(Z)\ni q:=p_C+p_D$ and define $\tl{q}(h):=q([h,h]_{\cE})^{1/2}$. 
Also note that $S(\cE)$ is closed under positive scalar multiplication. 

The following corollary is a first surrogate of a Schwarz type inequality.
\begin{corollary}\label{c:4schwarz}
Let $\cE$ be a topological VE-space over the topologically ordered space $Z$ and
$p\in S(Z)$. Then
\begin{equation*}p([e,f])\leq 4\, p([e,e])^{1/2}\,p([f,f])^{1/2},\quad e,f\in \cE.
\end{equation*}
\end{corollary}

\begin{proof} Firstly, for any $h,k\in \cE$, from \eqref{e:heka} and taking
  into account that $p\in S(Z)$ is increasing, it follows that
\begin{equation}\label{e:pehapka}p([h+k,h+k])\leq 2(p([h,h])+p([k,k]).
\end{equation}

Let now $e,f\in\cE$ be arbitrary. By the polarisation formula 
\eqref{e:polar} and \eqref{e:pehapka}, we have
\begin{align*}p([e,f]) & 
= p\bigl(\frac{1}{4}\sum_{k=0}^3 \iac^k [e+\iac^kf,e+\iac^kf]\bigr) \\
& \leq \frac{1}{4} \sum_{k=0}^3 p( [e+\iac^kf,e+\iac^kf] ) \\
& = \frac{1}{4} \sum_{k=0}^3 2\bigl(p([e,e])+p([\iac^kf,\iac^kf])\bigr) \\
& = 2 \bigl(p([e,e])+p([f,f])\bigr).
\end{align*}
Letting $\lambda>0$ arbitrary and changing  $e$ with
$\sqrt{\lambda} e$ and $f$ with $f/\sqrt{\lambda}$ in the previous inequality, we get
\begin{equation*}
p([e,f])\leq 2\bigl(\lambda p([e,e])+\lambda^{-1}p([f,f])\bigr) 
\end{equation*} hence, since the left hand side does not depend on $\lambda$,
it follows
\begin{equation*}p([e,f])\leq \inf_{\lambda>0}2\bigl(\lambda
  p([e,e])+\lambda^{-1}p([f,f])\bigr)  
=4\, p([e,e])^{1/2}\, p([f,f])^{1/2},
\end{equation*}
which is the required inequality.
\end{proof}

If $Z$ is an admissible space and $\cE$ is a topological VE-space whose
locally convex topology is complete, then $\cE$ is called a \emph{VH-space} 
(Vector Hilbert space). Any topological 
VE-space $\cE$ on an admissible space $Z$ 
can be embedded as a dense subspace of a VH-space $\cH$ over $Z$, 
uniquely determined up to an isomorphism, cf.\ Theorem 2 in
\cite{Loynes1}.

\begin{examples}\label{e:evhs}% examples VH-spaces
(1) Any Hilbert module $\cH$ over a $C^*$-algebra $\cA$, e.g.\ see \cite{Lance},
  \cite{ManuilovTroitsky}, can be viewed as a VH-space $\cH$ over the
  admissible space $\cA$, see Example~\ref{ex:as}.(1). In particular, any
  closed subspace $\cS$ of $\cH$ is a VH-space over the admissible space
  $\cA$.

(2) Any Hilbert module $\cH$ over a locally $C^*$-algebra $\cA$, e.g.\ see
\cite{Inoue}, \cite{Phillips}, can be viewed as a VH-space $\cH$ over the
admissible space $\cA$, see Example~\ref{ex:as}.(2). In particular, any
closed subspace $\cS$ of $\cH$ is a VH-space over the admissible space $\cA$.

(3) With notation as in Example~\ref{ex:as}.(5), consider $\cC_2$ the ideal of Hilbert-Schmidt
operators on $\cH$. Then $[A,B]=A^*B$, for all $A,B\in\cC_2$, is a gramian with values in the 
admissible space $\cC_1$ with respect to which $\cC_2$ becomes a VH-space. Observe that, 
since $\cC_1$ is a normed admissible space, by Lemma~\ref{l:topology} it follows that
$\cC_2$ is a normed VH-space, with norm 
$\|A\|_2=\tr(|A|^2)^{1/2}$, for all $A\in\cC_2$. More abstract versions of this example have been
considered by Saworotnow in \cite{Saworotnow}.

(4) Let $\{\cE_\alpha\}_{\alpha\in A}$ be a family of VH-spaces such that, for each
$\alpha\in A$, $\cE_\alpha$ is a VH-space over the admissible space
$Z_\alpha$. As in Example~\ref{ex:as}, consider the admissible space 
$Z=\prod_{\alpha\in A}Z_\alpha$ and the vector space $\cE=\prod_{\alpha\in A}
\cE_\alpha$ on which we define
\begin{equation*}[(e_\alpha)_{\alpha\in A},(f_\alpha)_{\alpha\in
      A}]=([e_\alpha,f_\alpha])_{\alpha\in A}\in Z,\quad 
(e_\alpha)_{\alpha\in A},(f_\alpha)_{\alpha\in A}\in \cE. 
\end{equation*} Then $\cE$ is a VE-space over $Z$. On $Z$ consider
the topology generated by the family of
increasing seminorms $\{q_j^{(\beta)}\}_{\substack{\beta\in
    A\\ j\in\cJ_\beta}}$ defined at \eqref{e:qs}, with respect to which 
$Z$ becomes an admissible space. For each $\beta\in A$ and each
$j\in\cJ_\beta$, in view of Lemma~\ref{l:topology}, consider the seminorm
\begin{equation*} \tl{q}_j^{(\beta)}((e_\alpha)_{\alpha\in
    A})=p_j^{(\beta)}([e_\alpha,e_\alpha])^{1/2}, \quad (e_\alpha)_{\alpha\in A}\in\cE.
\end{equation*}
The family of seminorms $\{\tl{q}_j^{(\beta)}\}_{\substack{\beta\in
    A\\ j\in\cJ_\beta}}$
generates on $\cE$ the topology with respect to which 
it is a VH-space over $Z$.

(5) Let $Z$ be an admissible space and $\cE_1,\ldots,\cE_n$ VH-spaces over
$Z$. On  $\cE=\prod_{j=1}^n \cE_j$ define
\begin{equation}
[(e_j)_{j=1}^n,(f_j)_{j=1}^n]_\cE=\sum_{j=1}^n [e_j,f_j]_{\cE_j},\quad
(e_j)_{j=1}^n,(f_j)_{j=1}^n\in\cE, 
\end{equation}
and observe that $(\cE;[\cdot,\cdot]_{\cE})$ is a VE-space over $Z$. In addition, for any $p\in S(Z)$
letting $\widetilde p\colon \cE\ra\RR_+$ be defined as in \eqref{e:qujeh}, 
$\widetilde p(e)=p([e,e]_\cE)^{1/2}$, for all $e\in\cE$, it is easy to see that $\cE$ 
is a VH-space over $Z$. It is clear that we can 
denote this VH-space by $\bigoplus_{j=1}^n \cE_j$ and call it the 
\emph{direct sum VH-space} of the VH-spaces $\cE_1,\ldots,\cE_n$.

(6) Let $\cH$ be a Hilbert space and $\cE$ a
VH-space over the admissible space $Z$. On the algebraic tensor product
$\cH\otimes \cE$ define a gramian by
\begin{equation*} [h\otimes e,l\otimes f]_{\cH\otimes\cE}
=\langle h,l\rangle_\cH [e,f]_\cE\in Z,\quad
  h,l\in \cH,\ e,f\in \cE,
\end{equation*} and then extend it to $\cH\otimes \cE$ by linearity. It can
be proven that, in this way, $\cH\otimes \cE$ is a VE-space over $Z$. Since
$Z$ is an admissible space, $\cH\otimes\cE$ can be topologised as in
Lemma~\ref{l:topology} and then completed to a VH-space
$\cH\widetilde\otimes\cE$ over $Z$.

If $\cH=\CC^n$ for some $n\in\NN$ then, with notation as in item (5), it is clear that 
$\CC^n\otimes\cE$ is isomorphic with $\bigoplus_{j=1}^n \cE_j$, with $\cE_j=\cE$ for all 
$j=1,\ldots,n$.
\end{examples}

\begin{remark}\label{r:continuity}
If $\cE$ and $\cF$ are two VH-spaces over the 
same admissible space $Z$, 
by $\cL_{c}(\cE,\cF)$ we denote 
the space of all continuous operators 
from $\cE$ to $\cF$. Let $\cC_0$ be a system of open  and absolutely convex 
neighbourhoods of the origin defining the topology of $Z$.
Since $S(\cE)$ is 
directed and it is 
closed under positive scalar multiplication, the continuity of a 
linear operator $T\in\cL(\cE,\cF)$ is equivalent with: 
for any $p\in S_{\cC_0}(\cF)$, there exists $q\in S(\cE)$ 
and a constant $c\geq 0$ such that 
$p(Th)\leq c\, q(h)$ for all $h\in\cE$. We will 
use this fact frequently in this article.\end{remark}

For $\cE$ and $\cF$ two VH-spaces over the same admissible space $Z$, we 
denote by $\cL_{c}^*(\cE,\cF)$ the subspace of $\cL^*(\cE,\cF)$ consisting 
of all continuous and continuously adjointable operators. 
Note that $\cL_{c}^*(\cE)=\cL_{c}^*(\cE,\cE)$ 
is an ordered $*$-subalgebra of $\cL^*(\cE)$. 

\begin{lemma}\label{l:redunt}%condition (b2) in t:vhinvkolmo2 is redundant.
Let $\cH$ be a topological VE-space over 
the topologically ordered $*$-space $Z$. Let $T\in\cL_{c}^*(\cH)$ be 
a positive operator and $p\in S(Z)$. Then there exist 
$q\in S(Z)$ and $c(T,p) \geq 0$ such that 
\begin{equation*}
p([Th,h]_{\cH})\leq c(T,p)\, q([h,h]_{\cH}),\quad h\in\cH.
\end{equation*}
\end{lemma}

\begin{proof}
To a certain extent, we use an argument in \cite{Loynes1}.
From
\begin{equation*}
[Th-h,Th-h]_{\cH}=[Th,Th]_{\cH}-2[Th,h]_{\cH}+[h,h]_{\cH}\geq0,
\end{equation*}
and taking into account that $T$ is positive, we obtain
\begin{equation*}0\leq 2[Th,h]_\cH\leq [Th,Th]_\cH+[h,h]_\cH.
\end{equation*}
From here, for any seminorm $p\in S(Z)$, using that $p$ is increasing, $T$ 
is continuous, and  Remark~\ref{r:continuity}, it follows that there exist
$q\in S(Z)$ and a constant $c(T,p)\geq 0$ such that, for all $h\in\cH$ we have
\begin{equation*}
p([Th,h]_{\cH})\leq \frac{1}{2}\bigl(p([Th,Th]_{\cH})+p([h,h]_{\cH})\bigr)
\leq c(T,p)\, q([h,h]_{\cH}).\qedhere
\end{equation*}
\end{proof} 

\begin{remark}\label{r:redunt} The previous lemma can be obtained as a
  consequence of the Schwarz type inequality as in Corollary~\ref{c:4schwarz}
  and the fact that $S(Z)$ is directed, but this is more involved than the
  presented proof.
\end{remark}

Let $\cH_1$ and $\cH_2$ be two VH spaces 
over the same admissible space $Z$, with their
family of seminorms $S(\cH_1)=\{\tl{p}_{\cH_1}\mid p\in S(Z)\}$ and, 
respectively, $S(\cH_2)=\{\tl{p}_{\cH_2}\mid p\in S(Z)\}$. 
Then the \emph{strict topology} on 
$\cL^*(\cH_1,\cH_2)$ is defined by the seminorms
$T \mapsto \tl{p}_{\cH_2}(T\xi)$ for 
$\tl{p}_{\cH_2} \in S(\cH_2)$, 
$\xi\in\cH_1$ and 
$T\mapsto \tl{p}_{\cH_1}(T^*\eta)$ for 
$\tl{p}_{\cH_1} \in S(\cH_1)$, 
$\eta \in \cH_2$, for all  
$p\in S(Z)$ with 
the seminorms $\tl{p}_{\cH_1}$ on $\cH_1$ and $\tl{p}_{\cH_2}$ on 
$\cH_2$ defined at \eqref{e:qujeh}. 
Equivalently, we can use all $p\in S_{\cC_0}(Z)$, where 
$\cC_0$ is a collection of open, absolutely convex, and solid neighbourhoods of
$0$ and that generates the topology of $Z$, as in Subsection \ref{ss:as}. 

\begin{lemma}\label{l:compstrict}%completeness of $\cL^*(\cH)$ with strict topology
Let $\cH_1$ and $\cH_2$ be two VH-spaces over 
the same admissible space $Z$. Then 
$\cL^*(\cH_1,\cH_2)$ with the strict topology is complete . 
\end{lemma}

\begin{proof}
Let $(T_i)_i$ be a Cauchy net in $\cL^*(\cH_1,\cH_2)$ with respect to the
strict topology. Then, $(T_{i}\xi)_i$ is a Cauchy net in 
$\cH_2$ for all $\xi\in\cH_1$ and 
$(T_{i}^*\eta)_i$ is a Cauchy net in $\cH_1$ for all 
$\eta \in \cH_1$, since 
they are Cauchy with respect to 
all seminorms in $S(\cH_2)$ and 
$S(\cH_1)$, respectively. Since 
$\cH_1$ and $\cH_2$ are complete, we have that 
$T_{i}\xi \xrightarrow[i]{} x_{\xi}$ and 
$T_{i}^*\eta \xrightarrow[i]{} y_{\eta}$ 
for some $x_{\xi} \in \cH_2$ and $y_{\eta} \in \cH_1$. 

Define the linear operators $T:\cH_1 \ra \cH_2$ 
by $T\xi=x_{\xi}$ and $R:\cH_2 \ra \cH_1$ 
by $R\eta=y_{\eta}$. Then, by the continuity of the gramians, see
Lemma~\ref{l:topology}, we have 
\begin{equation*}
[ T\xi,\eta ]_{\cH_2}=
\lim_{i}[ T_{i}\xi,\eta ]_{\cH_2}=
\lim_{i}[ \xi,T_{i}^*\eta ]_{\cH_1}=
[ \xi,R\eta ]_{\cH_1}.
\end{equation*}
Therefore, $T$ is adjointable with $T^*=R$, and $T_i \xrightarrow[i]{} T$ 
in the strict topology of $\cL^*(\cH_1,\cH_2)$. 
\end{proof}

A subspace $\cM$ of a VH-space $\cH$ is \emph{orthocomplemented}, or 
\emph{accessible} \cite{Loynes1}, 
if every element $h \in \cH$ can be written as 
$h=g+k$ where $g$ 
is in $\cM$ and $k$ is such that $[l,k]=0$ for all $l \in \cM$, 
that is, $k$ is in the 
\emph{orthogonal companion} $\cM^\perp$ of $\cM$. Observe that if such a 
decomposition exists it is unique and hence the \emph{orthogonal projection} 
$P_\cM$ onto $\cM$ can be defined by $P_\cM h=g$.
Any orthogonal projection $P$ is selfadjoint and idempotent, 
in particular we have 
$[Ph,k]=[Ph,Pk]$ for all $j,k\in \cH$, hence $P$ is positive and contractive,
in the sense $[Ph,Ph]\leq [h,h]$ for all $h\in\cH$, hence $P$ is continuous. 
Conversely, any selfadjoint idempotent operator is an orthogonal 
projection onto its range subspace. Any orthocomplemented subspace is closed.

\section{Positive Semidefinite Kernels with Values Adjointable
  Operators}\label{s:pskvao} 

Our main result is Theorem~\ref{t:vhinvkolmo2}  
that provides necessary and sufficient conditions
for a positive semidefinite kernel with values adjointable operators and
invariant under an action of a $*$-semigroup to give rise to a
$*$-representation of the given $*$-semigroup on a VH-space. We first provide
some preliminary results on positive semidefinite kernels with values
adjointable operators in a VE-space, cf.\ \cite{AyGheondea}.

\subsection{Kernels with Values Adjointable Operators.}\label{ss:kvao}
Let $X$ be a nonempty set and let $\cH$ be a VE-space over the ordered 
$*$-space $Z$. 
A map $\fk\colon X\times X\ra \cL(\cH)$ is called a \emph{kernel} on $X$ and 
valued in 
$\cL(\cH)$. In case the kernel $\fk$ has all its values in $\cL^*(\cH)$, 
an \emph{adjoint} kernel 
$\fk^*\colon X\times X\ra\cL^*(\cH)$ can be associated by 
$\fk^*(x,y)=\fk(y,x)^*$ for all 
$x,y\in X$. The kernel $\fk$ is called \emph{Hermitian} if $\fk^*=\fk$.

Let $\cF=\cF(X;\cH)$ denote the complex vector space of all functions 
$f\colon X\ra \cH$ and let $\cF_0=\cF_0(X;\cH)$ be its subspace of those 
functions having 
finite support. A pairing $[\cdot,\cdot]_{\cF_0}\colon \cF_0\times \cF_0\ra Z$ 
can be defined by
\begin{equation}\label{e:prodge}
[g,h]_{\cF_0}=\sum_{y\in X} [g(y),h(y)]_\cH,\quad g,h\in\cF_0.
\end{equation} This pairing is clearly a $Z$-gramian on $\cF_0$,  hence 
$(\cF_0;[\cdot,\cdot]_{\cF_0})$ is a VE-space. 

Another pairing $[\cdot,\cdot]_\fk$ can be defined on $\cF_0$ by
\begin{equation}\label{e:prodka} [g,h]_\fk=
\sum_{x,y\in X}[\fk(y,x)g(x),h(y)]_\cH,
\quad g,h\in \cF_0.\end{equation} In general, the pairing $[\cdot,\cdot]_\fk$ 
is linear in the second variable and conjugate linear in the first variable. 
If, in addition, 
$\fk=\fk^*$ then the pairing $[\cdot,\cdot]_\fk$ is Hermitian as well, that is,
\begin{equation*} [g,h]_\fk=[h,g]_\fk^*,\quad g,h\in\cF_0.
\end{equation*} 

A \emph{convolution operator} $K\colon\cF_0\ra\cF$ can be associated to the 
kernel $\fk$ by
\begin{equation}\label{e:convop} (Kg)(y)=\sum_{x\in X} \fk(y,x)g(x),\quad
  g\in\cF_0, 
\end{equation} and it is easy to see that $K$ is a linear operator.
There is a natural relation between the pairing $[\cdot,\cdot]_\fk$ 
and the convolution operator $K$ given by
\begin{equation*}[g,h]_\fk=[Kg,h]_{\cF_0},\quad g,h\in\cF_0.
\end{equation*} 

Given $n\in\NN$, the kernel $k$ is called \emph{$n$-positive} if for any 
$x_1,x_2,\ldots,x_n\in X$ and any $h_1,h_2,\ldots,h_n\in\cH$ we have
\begin{equation}\label{e:npos} \sum_{i,j=1}^n [\fk(x_i,x_j)h_j,h_i]_\cH\geq 0.
\end{equation}
 The kernel $k$ is called \emph{positive semidefinite} 
(or \emph{of positive type}) if it is $n$-positive for all natural numbers $n$.

 \begin{lemma}[Lemma~3.1 from \cite{Gheondea}]\label{L:twopos} 
Assume that the kernel $\fk\colon X\times 
X\ra\cL^*(\cH)$ is $2$-positive. Then:
 
 \nr{1} $\fk$ is Hermitian.
 
 \nr{2} If, for some $x\in X$, we have $\fk(x,x)=0$, then $\fk(x,y)=0$ for all 
$y\in X$.

\nr{3} There exists a unique decomposition $X=X_0\cup X_1$, such that
$X_0\cap X_1=\emptyset$, $\fk(x,y)=0$ for all $x,y\in X_0$ and $\fk(x,x)\neq
0$ for all $x\in X_1$.
\end{lemma}

Given an $\cL^*(\cH)$-valued kernel $\fk$ on a nonempty set $X$, for some 
VE-space $\cH$ on an ordered $*$-space $Z$,  a \emph{VE-space 
linearisation} or, equivalently, a
\emph{VE-space Kolmogorov decomposition} of $\fk$ is, by definition, 
a pair $(\cK;V)$, subject to the following conditions:
  
  \begin{itemize}
  \item[(vel1)] $\cK$ is a VE-space over the same ordered $*$-space $Z$.
  \item[(vel2)] $V\colon X\ra\cL^*(\cH,\cK)$ satisfies $\fk(x,y)=V(x)^*V(y)$ 
for all $x,y\in X$. \end{itemize}
The VE-space linearisation $(\cK;V)$ is called \emph{minimal} if
  \begin{itemize}
  \item[(vel3)] $\lin V(X)\cH=\cK$.
  \end{itemize}
Two VE-space linearisations $(V;\cK)$ and $(V';\cK')$ 
of the same kernel $\fk$ are 
called \emph{unitary equivalent} if there exists a VE-space isomorphism 
$U\colon \cK\ra\cK'$ such that $UV(x)=V'(x)$ for all $x\in X$.
  
The uniqueness of a minimal VE-space 
linearisation $(\cK;V)$ of a positive semidefinite kernel $\fk$, 
modulo unitary equivalence, 
follows in the usual way, see \cite{AyGheondea}.

Let $\cH$ be a VE-space over the ordered $*$-space $Z$, and let $X$ be a 
nonempty set. A VE-space $\cR$, 
over the same ordered $*$-space 
$Z$, is called an \emph{$\cH$-reproducing kernel VE-space on $X$} 
if there exists a Hermitian kernel $\fk\colon X\times X\ra\cL^*(\cH)$ 
such that the following axioms are satisfied:
\begin{itemize} 
\item[(rk1)] $\cR$ is a subspace of $\cF(X;\cH)$, with all algebraic operations.
\item[(rk2)] For all $x\in X$ and all $h\in\cH$, 
the $\cH$-valued function $\fk_x h=\fk(\cdot,x)h\in\cR$.
\item[(rk3)] For all $f\in\cR$ we have $[f(x),h]_\cH=[f,k_x h]_\cR$, for all 
$x\in X$ and $h\in \cH$.
\end{itemize}
As a consequence of (rk2), $\lin\{\fk_xh\mid x\in X,\ h\in\cH\}\subseteq \cR$.
The reproducing kernel VE-space $\cR$ is called \emph{minimal} if
the following property holds as well:
\begin{itemize}
\item[(rk4)] $\lin\{\fk_xh\mid x\in X,\ h\in \cH\}=\cR$.
\end{itemize}

Observe that if $\cR$ is an $\cH$-reproducing kernel VE-space on $X$ with 
kernel 
$\fk$, then $\fk$ is positive semidefinite and uniquely determined by $\cR$ 
hence, we can talk about \emph{the} 
$\cH$-reproducing kernel $\fk$ corresponding to $\cR$.
On the other hand, a minimal reproducing kernel VE-space $\cR$ is 
uniquely determined by its reproducing kernel $\fk$.

Letting $\cH$ be a VE-space over
an ordered $*$-space $Z$, for $X$ a nonempty set, an \emph{evaluation operator}
$E_x\colon \cF(X;\cH)\ra \cH$ can be defined for each $x\in X$ by letting 
$E_x f=f(x)$ for all $f\in\cF(X;\cH)$. Clearly, $E_x$ is linear.
If $\cR\subseteq \cF(X;\cH)$, with all algebraic operations, 
is a VE-space over $Z$, then $\cR$ 
is an $\cH$-reproducing kernel VE-space if and only if, for all $x\in X$, the 
restriction of the evaluation operator $E_x$ to $\cR$ is adjointable as a
linear operator $\cR\ra \cH$, e.g.\ see \cite{AyGheondea}.

\begin{proposition}[Proposition~2.4 in \cite{AyGheondea}]
\label{p:lvsrk}% linearisation vs reproducing kernels
Let $X$ be a nonempty set, $\cH$ a VE-space over an ordered $*$-space
$Z$, and let $\fk\colon X\times X\ra\cL^*(\cH)$ be a Hermitian kernel.

\nr{1} Any $\cH$-reproducing kernel VE-space $\cR$ with kernel $\fk$ is
a VE-space linearisation $(\cR;V)$ of $\fk$, with $V(x)=\fk_x$ for all $x\in X$.

\nr{2} For any minimal VE-space linearisation $(\cK;V)$ of $\fk$, letting 
\begin{equation}\label{e:redef} \cR=\{V(\cdot)^*f\mid f\in\cK\},
\end{equation} we obtain an $\cH$-reproducing kernel VE-space with
reproducing kernel $\fk$.
\end{proposition}

Let a (multiplicative) semigroup $\Gamma$ act on $X$, 
denoted by $\xi\cdot x$, for all $\xi\in\Gamma$ 
and all $x\in X$. By definition, we have 
\begin{equation}\label{e:action}
\alpha\cdot(\beta\cdot x)=(\alpha\beta)\cdot x\mbox{ for all }\alpha,
\beta\in \Gamma\mbox{ and all }x\in X.\end{equation} 
Equivalently, this means that we have a semigroup morphism 
$\Gamma\ni\xi\mapsto \xi\cdot \in G(X)$, where $G(X)$ denotes the 
semigroup, with respect to composition, of all maps $X\ra X$. 
In case the semigroup $\Gamma$ has a unit 
$\epsilon$, the action is called \emph{unital} if $\epsilon\cdot x=x$ 
for all $x\in X$, equivalently, $\epsilon\cdot=\mathrm{Id}_X$.

Assume that $\Gamma$ is a $*$-semigroup, that is, there is an 
\emph{involution} $*$ on $\Gamma$: $(\xi\eta)^*=\eta^* \xi^*$ 
and $(\xi^*)^*=\xi$ for all $\xi,\eta\in\Gamma$. 
Note that, in case $\Gamma$ has a unit $\epsilon$ then 
$\epsilon^*=\epsilon$.

Given a VE-space $\cH$ we consider those Hermitian kernels 
$\fk\colon X\times X\ra\cL^*(\cH)$ that are \emph{invariant} under the action 
of $\Gamma$ on $X$, that is,
\begin{equation}\label{e:invariant} \fk(y,\xi\cdot x)=\fk(\xi^*\cdot y,x)
\mbox{ for all }x,y\in X\mbox{ and all }\xi\in\Gamma.
\end{equation}
A triple $(\cK;\pi;V)$ is called an \emph{invariant VE-space linearisation} 
of  the kernel $\fk$ and the action of $\Gamma$ on $X$, shortly a
\emph{$\Gamma$-invariant VE-space linearisation} of $\fk$, if:
\begin{itemize}
\item[(ikd1)] $(\cK;V)$ is a VE-space linearisation of the kernel $\fk$.
\item[(ikd2)] $\pi\colon \Gamma\ra\cL^*(\cK)$ is a $*$-representation, that
  is, a multiplicative $*$-morphism.
\item[(ikd3)] $V$ and $\pi$ are related by the formula: 
$V(\xi\cdot x)= \pi(\xi)V(x)$, for all $x\in X$, $\xi\in\Gamma$.
\end{itemize}

If $(\cK;\pi;V)$ is a $\Gamma$-invariant VE-space 
linearisation of the kernel 
$\fk$ then $\fk$ is invariant under the action of $\Gamma$ on $X$.

If, in addition to the axioms (ikd1)--(ikd3), the triple $(\cK;\pi;V)$ 
has the property
\begin{itemize}
\item[(ikd4)] $\lin V(X)\cH=\cK$,
\end{itemize} that is, the VE-space linearisation $(\cK;V)$ is minimal,
then $(\cK;\pi;V)$ is called a \emph{minimal
$\Gamma$-invariant VE-space linearisation} of $\fk$ and the action of 
$\Gamma$ on $X$.

\begin{theorem}[Theorem~2.8 in \cite{AyGheondea}]\label{t:vhinvkolmo} 
Let $\Gamma$ be a $*$-semigroup 
that acts 
on the nonempty set $X$ and let $\fk\colon X\times X\ra\cL^*(\cH)$ be a kernel, 
for some VE-space $\cH$ over an ordered $*$-space $Z$. 
The following assertions are equivalent:

\begin{itemize}
\item[(1)] $\fk$ is positive semidefinite, in the sense of \eqref{e:npos}, 
and invariant under the action of $\Gamma$ on $X$, that is, 
\eqref{e:invariant} holds.
\item[(2)] $\fk$ has a $\Gamma$-invariant VE-space linearisation 
$(\cK;\pi;V)$.
\item[(3)] $\fk$ admits an $\cH$-reproducing kernel VE-space $\cR$ and 
there exists a $*$-representation $\rho\colon \Gamma\ra\cL^*(\cR)$ such that 
$\rho(\xi)\fk_xh=\fk_{\xi\cdot x}h$ for all $\xi\in\Gamma$, $x\in X$, $h\in\cH$.
\end{itemize}

In addition, in case any of the assertions \emph{(1)}, \emph{(2)}, or
\emph{(3)} holds,  
then a minimal\, $\Gamma$-invariant VE-space linearisation can be constructed,
any minimal $\Gamma$-invariant VE-space linearisation is unique up to unitary 
equivalence, 
a pair $(\cR;\rho)$ as in assertion \emph{(3)} with $\cR$ 
minimal can be always obtained and, in this case, it is uniquely 
determined by $\fk$ as well.
\end{theorem}

Because we will use some of the constructions provided by the proof of 
Theorem~\ref{t:vhinvkolmo} we recall those needed.
Assuming that $\fk$ is positive semidefinite, by 
Lemma~\ref{L:twopos}.(1) it follows that $\fk$ is Hermitian, that is, 
$\fk(x,y)^*=\fk(y,x)$ for all $x,y\in X$.
We consider the convolution 
operator $K$ defined at \eqref{e:convop} and let $\cG=\cG(X;\cH)$ be its 
range, more precisely,
\begin{align}\label{e:fezero} \cG & =\{f\in\cF\mid f=Kg\mbox{ for some }
g\in\cF_0\} \\
& = \{f\in\cF\mid f(y)=\sum_{x\in X} \fk(y,x)g(x)\mbox{ for some }g\in\cF_0
\mbox{ and all } 
x\in X\}.\nonumber
\end{align}
A pairing $[\cdot,\cdot]_{\cG}\colon \cG\times \cG\ra Z$ can be defined by
\begin{align}\label{e:ipfezero} 
[e,f]_{\cG} &=[Kg,h]_{\cF_0}=\sum_{y\in X}[e(y),h(y)]_{\cH}
 =
\sum_{x,y\in X} [\fk(y,x)g(x),h(y)]_\cH,
\end{align} where $f=Kh$ and $e=Kg$ for some $g,h\in\cF_0$. 
The pairing $[\cdot,\cdot]_{\cG}$ is a $Z$-valued gramian, that is, 
it satisfies all the requirements (ve1)--(ve3). 
$(\cG;[\cdot,\cdot]_{\cG})$ is a VE-space that we denote by $\cK$.
For each $x\in X$ define $V(x)\colon \cH\ra\cG$ by
\begin{equation}\label{e:defvex} V(x) h=Kh_x,\quad h\in\cH,
\end{equation} where $h_x=\delta_x h\in\cF_0$ is the function that takes the
value $h$ at $x$ and is null elsewhere.
Equivalently,
\begin{equation}\label{e:vexa} (V(x) h)(y)=(Kh_x)(y)
=\sum_{z\in X} \fk(y,z)(h_x)(z)=\fk(y,x)h,\ \ y\in X.
\end{equation}
Note that $V(x)$ is an operator from the VE-space $\cH$ to the 
VE-space $\cG=\cK$ and it can be shown that $V(x)$ is adjointable 
for all $x\in X$. 

On the other hand, for any $x,y\in X$, by
\eqref{e:vexa}, we have
\begin{equation*} V(y)^*V(x)h=(V(x)h)(y) =\fk(y,x)h,\quad h\in\cH,
\end{equation*} hence $(V;\cK)$ is a VE-space linearisation of $\fk$ and 
it is minimal as well, more precisely, $\cG$ is the range of the
convolution operator $K$ defined at \eqref{e:convop}.

For each $\xi\in\Gamma$ let 
$\pi(\xi)\colon \cF\ra\cF$ be 
defined by
\begin{equation}\label{e:dop}%definition of pi
(\pi(\xi)f)(y)=f(\xi^*\cdot y),\quad f\in\cF,\ y\in X,\ \xi\in\Gamma.
\end{equation}
 $\pi(\xi)$ leaves $\cG$ invariant. 
Denote by 
the same symbol $\pi(\xi)$ the map $\pi(\xi)\colon \cG\ra\cG$.

$\pi$ is a $*$-representation of the semigroup $\Gamma$ on 
the complex vector space $\cG$ and, taking into account 
that $\fk$ is 
invariant under the action of $\Gamma$ on $X$, 
for all $\xi\in\Gamma$, $x,y\in X$, $h\in\cH$,
we have
\begin{equation}\label{e:vexic} 
(V(\xi\cdot x)h)(y)  =\fk(y,\xi\cdot x)h=\fk(\xi^*\cdot y,x)h
  =(V(x)h)(\xi^*\cdot y)=(\pi(\xi) V(x)h)(y),
\end{equation} which proves (ikd3). Thus, $(\cK;\pi;V)$, here constructed, 
is a $\Gamma$-invariant VE-space linearisation of the Hermitian kernel $\fk$.
Note that $(\cK;\pi;V)$ is minimal, that is, the axiom (ikd4) holds, 
since the VE-space linearisation
$(\cK;V)$ is minimal.
%%%%%%%%%%%%%%%%%%%%%%%%%%%%%%%%%%%%

The construction of $(\cK;\pi;V)$ just presented
is essentially a minimal $\cH$-reproducing kernel VE-space 
one. In particular, it proves the statement (3) as well. On the other hand,
Proposition~\ref{p:lvsrk}
provides an explicit connection between the collection of all minimal
$\Gamma$-invariant VE-space
linearisations $(\cK;\pi;V)$ of $\fk$, identified by unitary equivalence, 
and the unique  minimal $\cH$-reproducing kernel VE-space $\cR$ of $\fk$. 
On $\cR$ a $Z$-valued gramian is defined by
\begin{equation}\label{e:ufege}
[V(\cdot)^*f,V(\cdot)^*g]_\cR=[f,g]_\cK,\quad f,g\in\cK.
\end{equation} 

\subsection{VE-Module Linearisations.}\label{ss:veml}
Given an ordered $*$-algebra $\cA$ and a VE-module $\cE$ over $\cA$, 
an \emph{$\cE$-reproducing kernel VE-module over $\cA$} is just an
$\cE$-reproducing kernel VE-space over $\cA$, with definition as in 
Subsection~\ref{ss:kvao}, which is also a VE-module over $\cA$. 
 
 \begin{proposition}\label{p:veinvkolmomodule} 
 Let $\Gamma$ be a $*$-semigroup that acts 
on the nonempty set $X$ and let $\fk\colon X\times X\ra\cL^*(\cH)$ 
be a kernel, for some VE-module $\cH$ over an ordered $*$-algebra $\cA$. 
The following assertions are equivalent:

\begin{itemize}
\item[(1)] $\fk$ is positive semidefinite, in the sense of \eqref{e:npos}, 
and invariant under the action of $\Gamma$ on $X$, that is, 
\eqref{e:invariant} holds.
\item[(2)] $\fk$ has a $\Gamma$-invariant VE-module (over $\cA$) 
linearisation $(\cK;\pi;V)$.
\item[(3)] $\fk$ admits an $\cH$-reproducing kernel VE-module $\cR$ and 
there exists a $*$-representation $\rho\colon \Gamma\ra\cL^*(\cR)$ such that 
$\rho(\xi)\fk_xh=\fk_{\xi\cdot x}h$ for all $\xi\in\Gamma$, $x\in X$, $h\in\cH$.
\end{itemize}

In addition, in case any of the assertions \emph{(1)}, \emph{(2)}, or
\emph{(3)} holds,  
then a minimal $\Gamma$-invariant VE-module 
linearisation can be constructed,
any minimal $\Gamma$-invariant VE-module 
linearisation is unique up to unitary equivalence, 
a pair $(\cR;\rho)$ as in assertion \emph{(3)} with $\cR$ 
minimal can be always obtained and, in this case, it is uniquely 
determined by $\fk$ as well.
\end{proposition}

We briefly recall the construction made in the implication 
(1)$\Ra$(2), for later use. We first observe that, 
 since $\cH$ is a module over $\cA$,
 the space $\cF(X;\cH)$ has a natural structure of right module over 
 $\cA$, more precisely, for any $f\in\cF(X;\cH)$ and $a\in \cA$
 \begin{equation*} (fa)(x)= f(x)a,\quad x\in X.
 \end{equation*}
 In particular, the space $\cF_0(X;\cH)$ is a submodule of $\cF(X;\cH)$. On
 the other hand, by assumption, for each $x,y\in X$, $\fk(x,y)\in\cL^*(\cH)$, 
 hence $\fk(x,y)$ is a module map. These imply that the convolution operator
 $K\colon\cF_0(X;\cH)\ra\cF(X;\cH)$ defined as in \eqref{e:convop} is a 
module map. Indeed, for any $f\in\cF_0(X;\cH)$, $a\in\cA$, and $y\in X$,
\begin{equation*} ((Kf)a)(x)=\sum_{x\in X}\fk(y,x)f(x)a=K(fa)(x).
\end{equation*}
Then, the space $\cG(X;\cH)$ which, with the definition as in 
\eqref{e:fezero}, is the range of the convolution operator $K$, 
is a module over $\cA$ as well. 

When endowed with the $\cA$ valued gramian $[\cdot,\cdot]_\cG$ defined
as in \eqref{e:ipfezero}, we have
\begin{equation}\label{e:gmp}%gramian module property
[e,fa]_\cG=[e,f]_\cG\, a,\quad e,f\in\cG(X;\cH),\ a\in\cA.
\end{equation}
Indeed, let $e=Kg$ and $f=Kh$ for some $g,h\in\cF_0(X;\cH)$. Then,
\begin{equation*}[e,fa]_\cG=[Kg,ha]_{\cF_0}=\sum_{y\in X}
[e(y),h(y)a]_\cH=\sum_{y\in X}[e(y),h(y)]_\cH a=[Kg,h]_{\cF_0}a=[e,f]_\cG a.
\end{equation*}

From \eqref{e:gmp} and the proof of the implication (1)$\Ra$(2) in 
Theorem~\ref{t:vhinvkolmo}, it follows that $\cK=\cG(X;\cH)$ is a 
VE-module over the ordered $*$-algebra $\cA$ and hence, the triple
$(\cK;\pi;V)$ is a minimal $\Gamma$-invariant VE-module linearisation of 
$\fk$.

\subsection{VH-Space Linearisations and Reproducing Kernels.}\label{ss:vhslrk}
Let $\cH$ be a VH-space over the admissible space $Z$, and 
consider a kernel $\fk\colon X\times X\ra \cL_c^*(\cH)$. A \emph{VH-space 
linearisation} of $\fk$, or  
\emph{VH-space Kolmogorov decomposition} of $\fk$, is
a pair $(\cK;V)$, subject to the following conditions:
  
  \begin{itemize}
  \item[(vhl1)] $\cK$ is a VH-space over the same ordered $*$-space $Z$.
  \item[(vhl2)] $V\colon X\ra\cL^*_c(\cH,\cK)$ satisfies $\fk(x,y)=V(x)^*V(y)$ 
for all $x,y\in X$. \end{itemize}
The VH-space linearisation $(\cK;V)$ is called \emph{minimal} if
  \begin{itemize}
  \item[(vhl3)] $\lin V(X)\cH$ is dense in $\cK$.
  \end{itemize}
It is useful to observe that any VH-space linearisation is a VE-space
linearisation with some differences 
between them: the former requires that both the kernel $\fk$ and all
the operators $V(x)$, $x\in X$, are all continuous and continuously 
adjointable operators. As concerning minimality, the two concepts 
are significantly different. 

Two VH-space linearisations $(V;\cK)$ and $(V';\cK')$ 
of the same kernel $\fk$ are 
called \emph{unitary equivalent} if there exists a unitary operator 
$U\colon \cK\ra\cK'$ such that $UV(x)=V'(x)$ for all $x\in X$.
  
The uniqueness of a minimal VH-space 
linearisation $(\cK;V)$ of a positive semidefinite kernel $\fk$, 
modulo unitary equivalence, 
follows in the usual way, taking into account that unitary operators are 
continuous, e.g.\ see \cite{Gheondea}.

A VH-space $\cR$ over the ordered $*$-space 
$Z$ is called an \emph{$\cH$-reproducing kernel VH-space on $X$} 
if there exists a Hermitian kernel $\fk\colon X\times X\ra\cL^*_c(\cH)$ 
such that the following axioms are satisfied:
\begin{itemize} 
\item[(rk1)] $\cR$ is a subspace of $\cF(X;\cH)$, with all algebraic operations.
\item[(rk2)] For all $x\in X$ and all $h\in\cH$, 
the $\cH$-valued function $\fk_x h=\fk(\cdot,x)h\in\cR$.
\item[(rk3)] For all $f\in\cR$ we have $[f(x),h]_\cH=[f,\fk_x h]_\cR$, for all 
$x\in X$ and $h\in \cH$.
\item[(rk4)] For all $x\in X$ the evaluation operator $\cR\ni f\mapsto f(x)\in
\cH$ is continuous.
\end{itemize}
Note that, when comparing a reproducing kernel VH-space with a reproducing
kernel VE-space, for the same kernel $\fk$, there are at least two
differences. First, in
the former, we have a VH-space and the values of the kernel are all continuous 
and continuously adjointable
operators. Second, the axiom (rk4) is new even when compared to the
classical case of reproducing kernel Hilbert spaces, when this is actually a
consequence of the other axioms. 
As the following result shows, these differences have
consequences that differentiate the reproducing kernel 
VH-space from the reproducing kernel VE-space and from the reproducing kernel
Hilbert space.

\begin{lemma}\label{l:mincont}% minimality and continuity
Let $\cR$ be an $\cH$-reproducing kernel VH-space
with reproducing kernel $\fk$.

\emph{(1)} For any $x\in X$, $\fk_x\in\cL_c^*(\cH,\cR)$.

\emph{(2)} For any $x,y\in X$, $\fk(x,y)=\fk_x^* \fk_y$.

\emph{(3)} $\fk$ is positive semidefinite.

\emph{(4)} The orthogonal space of 
$\lin\{\fk_xh\mid x\in X,\ h\in\cH\}\subseteq \cR$ is the null space.

\emph{(5)} $\fk$ is uniquely determined by $\cR$.
\end{lemma}

\begin{proof} Clearly, for arbitrary $x\in X$, the map
$\fk_x\colon\cH\ra \cR$ is a linear operator. From (rk3) it follows that
  $\fk_x$ is adjointable and its adjoint $\fk_x^*$ is $E_x\colon \cR\ra\cH$, 
the evaluation operator $E_x(f)=f(x)$, for $f\in\cR$ which, by (rk4), 
is assumed to be continuous. On the other hand,  by (rk3),
for arbitrary $x,y\in X$, we have
\begin{equation*}
  [\fk_xh,\fk_yg]_\cR=[(\fk_xh)(y),g]_\cH=[\fk(y,x)h,g]_\cH,\quad h,g\in\cH,
\end{equation*} hence the assertion (2) is proven. In particular,
$\fk(x,x)=\fk_x^* \fk_x$ is a positive operator. Since, by assumption,
$\fk(x,x)\in \cL_c^*(\cH)$, we can apply Lemma~\ref{l:redunt} and obtain that,
for every seminorm $p\in S(Z)$ there exist a seminorm $q\in S(Z)$ and a
constant $c\geq 0$ such that
\begin{equation*}p([\fk_x h,\fk_x h]_\cR)=p([\fk(x,x)h,h]_\cH)\leq c\,
  q([h,h]),\quad h\in\cH,
\end{equation*} hence $\fk_x$ is continuous. This concludes the proof of
assertion (1).

Let $n\in\NN$, $x_1,\ldots,x_n\in X$, and $h_1,\ldots,h_n\in\cH$ be
arbitrary. Then
\begin{align*}\sum_{j,k=1}^n[\fk(x_j,x_k)h_k,h_j]_\cH & =
  \sum_{j,k=1}^n[\fk_{x_j}^*\fk_{x_k} h_k,h_j]_\cH=[\sum_{k=1}^n \fk_{x_k}h_k,
    \sum_{j=1}^n \fk_{x_j}h_j]_\cR\geq 0,
\end{align*} hence assertion (3) is proven.

Let $f\in\cR$ be an $\cH$-valued function orthogonal to all $\cH$-valued
functions $\fk_xh$, with $x\in X$ and $h\in \cH$. By (rk3), for each $x\in X$,
\begin{equation*}0=[f,\fk_x h]_\cR=[f(x),h]_\cH,\quad h\in\cH,
\end{equation*} and hence, since the gramian $[\cdot,\cdot]_\cH$ is
nondegenerate it follows that $f(x)=0$. Therefore, $f=0$ and assertion (4)
is proven as well.

In order to see that assertion (5) is true, observe that once the
$\cH$-reproducing kernel VH-space $\cR$ on the set $X$ is given, 
all the evaluation operators $E_x$ are uniquely determined by $\cR$. Since
$E_x=\fk_x^*$, from (2) it follows
\begin{equation*} \fk(y,x)=\fk_y^*\fk_x=E_y E_x^*,\quad x,y\in X,
\end{equation*} hence the kernel $\fk$ is uniquely determined by $\cR$.
\end{proof}

Assertion (4) in the previous lemma says that reproducing kernel VH-spaces
have a built-in minimality property but, due to the fact that not any closed
subspace of a VH-space is orthocomplemented, the following definition makes
sense. An $\cH$-reproducing kernel VH-space $\cR$ on $X$ is called
\emph{minimal} if

\begin{itemize}
\item[(5)] $\lin\{\fk_xh\mid x\in X,\ h\in\cH\}$ is dense in $\cR$.
\end{itemize}

\begin{proposition}\label{p:mrkvhs}% minimal reproducing kernel VH-space
Let $\cH$ be a VH-space over some admissible space $Z$ and $\fk$ an
$\cH$-kernel on $X$ and assume that there exists an $\cH$-reproducing kernel
VH-space $\cK$ on $X$ with reproducing kernel $\fk$.

\nr{1} The closure of $\lin\{\fk_xh\mid x\in X,\ h\in\cH\}$ in
$\cK$ is a minimal $\cH$-reproducing kernel VH-space on $X$ with kernel $\fk$.

\nr{1} If $\cR$ is another minimal $\cH$-reproducing kernel VH-space 
on $X$ with the same reproducing kernel $\fk$, then $\cR\subseteq \cK$. 
In particular, the
minimal $\cH$-reproducing kernel VH-space on $X$ with reproducing kernel $\fk$
is unique.
\end{proposition}

\begin{proof} (1) This statement is clear from the axioms (rk1)--(rk4).

(2) Clearly, $\cL=\lin\{\fk_xh\mid x\in X,\ h\in\cH\}$ is contained
in both $\cR$ and $\cK$. In addition,
\begin{equation*}[f,g]_\cR=[f,g]_{\cK},\quad f,g\in\cL,
\end{equation*} and, with notation as in Lemma~\ref{l:topology}, we have
$\tl{p}_\cR|\cL
=\tl{p}_\cK|\cL$ and hence $\tau_\cR|\cL=\tau_\cK|\cL$. 
By the minimality of $\cR$,
for any $f\in\cR$ there exists a net $(f_i)_i$ in $\cL$ such that
$f_i\xrightarrow[i]{\tau_\cR} f$ and
\begin{equation*} [f(x),h]_\cH=[f,\fk_xh]_\cR=\lim_i [f_i,\fk_x h]_\cR,\quad
  x\in X,\ h\in\cH.
\end{equation*} But, $(f_i)_i$ is a Cauchy net in
$(\cL;\tau_\cR|\cL)=(\cL;\tau_\cK|\cL)$ and hence, there exists $g\in\cK$ such
that $f_i\xrightarrow[i]{\tau_\cK} g$, which implies
\begin{equation*}[g(x),h]_\cH=[g,\fk_xh]_\cK=\lim_i [f_i,\fk_xh]_\cK.
\end{equation*} Since, for arbitrary fixed $x\in X$ and $h\in\cH$ we have
\begin{equation*}[f_i,\fk_xh]_\cK=[f_i,\fk_xh]_\cR,\quad \mbox{ for any }i,
\end{equation*} taking into account that $Z$ is separated, it follows
\begin{equation*}[f(x),h]_\cH=[g(x),h]_\cH,\quad x\in X,\ h\in\cH,
\end{equation*} hence $f=g\in\cR$. This proves $\cR\subseteq\cK$.
\end{proof}

Observe that, given $X$ a nonempty set and $\cH$ a VH-space, for any $x\in X$
one can define a
general \emph{evaluation operator} $E_x\colon \cF(X;\cH)\ra \cH$ by
$E_x(f)=f(x)$, for all $f\in\cF(X;\cH)$. In particular, evaluation operators
can be defined if instead of $\cF(X;\cH)$ we can consider any vector 
subspace $\cS$ of $\cF(X;\cH)$.  

\begin{proposition}\label{p:rks}
Let $X$ be a nonempty set, $\cH$ a VH-space over an admissible space 
$Z$, and let $\cR\subseteq \cF(X;\cH)$, with all algebraic operations, 
be a VH-space over $Z$. Then $\cR$ 
is an $\cH$-reproducing kernel VH-space if and only if, for all $x\in X$, the 
evaluation operator $E_x\in\cL_c^*(\cR,\cH)$, that is, $E_x$ is continuous and
continuously adjointable.
\end{proposition}

\begin{proof} Assume first that $\cR$ is an $\cH$-reproducing kernel VH-space
on $X$ and let $\fk$ be its reproducing kernel.
For any $h\in\cH$ and any $f\in\cR$ we have
\begin{equation}\label{e:aeo}%adjointable evaluation operator
[E_xf,h]_\cH=[f(x),h]_\cH=[f,\fk_x h]_\cR.
\end{equation}
Since $\fk_x \in\cL(\cH,\cR)$, it follows that $E_x$ is adjointable and, in 
addition, $E_x^*=\fk_x$, for all $x\in X$.  It was
proven in Lemma~\ref{l:mincont} that 
$\fk_x\in\cL_c^*(\cH,\cR)$, hence $E_x\in\cL_c^*(\cR,\cH)$.

Conversely, assume that, for all $x\in X$, the evaluation operator 
$E_x\in\cL^*_c(\cR,\cH)$. Equation \eqref{e:aeo} 
suggests to define the kernel $\fk$ 
in the following way:
\begin{equation}\label{e:drk}%definition of reproducing kernel
\fk(y,x)h=(E_x^*h)(y),\quad x,y\in X,\ h\in\cH.
\end{equation}
Then $k(y,x)\colon \cH\ra\cH$ is a linear operator and, letting
$\fk_x=\fk(\cdot,x)$ for all $x\in X$, we have 
$\fk_x h=E_x ^*h$ for all $x\in X$ and all $h\in \cH$. 
The reproducing property (rk3) holds:
\begin{equation*} [f(x),h]_\cH=[E_x f,h]_\cH=[f,E_x^* h]_\cR=[f,\fk_xh]_\cR,
\quad f\in\cR,\ h\in\cH,\ x\in X.
\end{equation*}
The axioms (rk1), (rk2), and (rk3) are clearly satisfied. We prove
that $\fk$ is a Hermitian kernel. To see this,
fix $x,y\in X$ and $h,l\in\cH$. Then
\begin{align*}[\fk(y,x)h,l]_\cH & = [(\fk_x h)(y),l]_\cH =[\fk_x h,\fk_y l]_\cR\\
& = [\fk_y l,\fk_x h]_\cR^*=[\fk(x,y)l,h]_\cR^*=[h,\fk(x,y)l]_\cR.
\end{align*}
Therefore, $\fk(y,x)$ is adjointable and $\fk(y,x)^*=\fk(x,y)$, hence $\fk$ is
a Hermitian kernel. We have proven that $\fk$ is the reproducing 
kernel of $\cR$.
\end{proof}

There is a very close connection between VH-space linearisations and
reproducing kernel VH-spaces, similar, to a certain extent, 
to the connection between VE-space linearisations and reproducing 
kernel VE-spaces, as in Proposition~\ref{p:lvsrk}.

\begin{proposition}\label{p:vhlvsrk}% linearisation vs reproducing kernels
Let $X$ be a nonempty set, $\cH$ a VH-space over an admissible space
$Z$, and let $\fk\colon X\times X\ra\cL_c^*(\cH)$ be a Hermitian kernel.

\nr{1} For any VH-space linearisation $(\cK;V)$ of $\fk$, letting $\cK_0$
denote the closure of the linear span of $V(X)\cH$ in $\cK$ and
$V_0(x)h:=V(x)h\in\cK_0$, for all $x\in X$ and all $h\in\cH$, we obtain a
minimal VH-space linearisation $(\cK_0;V_0)$ of $\fk$.

\nr{2} For any minimal VH-space linearisation $(\cK;V)$ of $\fk$, letting 
\begin{equation}\label{e:rehdef} \cR=\{V(\cdot)^*f\mid f\in\cK\},
\end{equation} we obtain the minimal $\cH$-reproducing kernel VH-space with
reproducing kernel $\fk$.

\nr{3} Any $\cH$-reproducing kernel VH-space $\cR$ with kernel $\fk$ is
a VH-space linearisation $(\cR;V)$ of $\fk$, with $V(x)=\fk_x$ for all $x\in X$.
In addition, the $\cH$-reproducing kernel VH-space $\cR$ is minimal if and
only if the VH-space linearisation $(\cR;V)$ is minimal.
\end{proposition}

\begin{proof} (1) Clearly $\cK_0$ is a VH-subspace of $\cK$. By its very
  definition, $V_0(x)\in \cL_c(\cH,\cK_0)$, for all $x\in X$. Fixing $x\in X$,
  we consider the linear operator $W(x)=V(x)^*|\cK_0\colon \cK_0\ra\cH$ and
  observe that $W(x)\in\cL_c(\cK_0,\cH)$. Then,
\begin{equation*}
  [W(x)k,h]_{\cH}=[V(x)^*k,h]_{\cH}=[k,V(x)h]_{\cK}=[k,V(x)h]_{\cK_o}
  =[k,V_0(x)]_{\cK_0},\quad h\in\cH,\ k\in\cK_0,
\end{equation*}
hence $W(x)$ is the adjoint operator of $V_0(x)$, hence
$V_0(x)\in\cL_c^*(\cH,\cK_0)$. In addition,
\begin{align*}[V_0(x)^*V_0(y)h,g]_{\cH} & =[V_0(y)h,V_0(x)g]_{\cK_0}
=[V(y)h,V(x)g]_{\cK}\\
& =[V(x)^*V(y)h,g]_{\cH}=[\fk(x,y)h,g]_{\cH},\quad
  h,g\in\cH, 
\end{align*}
hence $(\cK_0;V_0)$ is a VH-space linearisation of $\fk$. Since, by
definition, $\cK_0$ coincides with the closure of the linear span of
$V_0(X)\cH$, it follows that it is minimal as well.

(2) Let $(\cK;\pi;V)$ be a minimal VH-space linearisation of the kernel $\fk$.
Define $\cR$ as in \eqref{e:rehdef} that is, $\cR$ consists of all functions 
$X\ni x\mapsto V(x)^*f\in\cH$, in particular $\cR\subseteq \cF(X;\cH)$, 
and we endow $\cR$ with the algebraic operations inherited from the 
complex vector space $\cF(X;\cH)$. We consider the correspondence
\begin{equation}\label{e:defu} \cK\ni f\mapsto Uf=V(\cdot)^*f\in\cR.
\end{equation} 
From Proposition~\ref{p:lvsrk}, we know
that $(\cR; [ \cdot, \cdot ]_{\cR})$ with the $Z$-gramian 
$[ Uf, Ug ]_{\cR} = [ f, g ]_{\cK}$ is a VE-space, that
$U \colon \cK \ra \cR$ is a unitary operator of VE-spaces $\cK$ and $\cR$, 
and that $(\cR;[\cdot,\cdot]_{\cR})$
is an $\cH$-reproducing kernel VE-space with reproducing kernel $\fk$.
In addition, by \eqref{e:defu} and the definition of the natural topology
of a VH-space, see Lemma~\ref{l:topology}, it follows that $U$ is a
homeomorphism, hence $\cR$ is a VH-space. Therefore, the axioms (rk1)-(rk3)
hold and the minimality of $\cR$ follows from the minimality of $\cK$.  
It only remains  to show that the axiom (rk4) holds as well.

We show that $\fk_x \in \cL_{c}^{*}(\cH, \cR) $ for all $x \in X$. First
recall that $\fk_x \in \cL^{*}(\cH, \cR) $ for all $x \in X$
by the reproducing kernel axiom. We first prove that
$\fk_x$ is continuous. By the continuity
of $V(x)$ for arbitrary $x \in X$, for any 
$p \in S(Z)$ there exist $q \in S(Z)$ and $c_p(x) \geq 0$ such that
\begin{equation*}
p ([ \fk_{x}(h), \fk_{x}(h) ]_{\cR})
= p ([ \fk_{x}h, \fk_{x}h ]_{\cR})
=p ([ V(x)h, V(x)h ]_{\cK}) 
\leq c_p(x)\,q ( [ h, h ]_{\cH} ),\quad h\in\cH,
\end{equation*}
hence $\fk_x$ is continuous.

Finally we show that 
$\fk_x^{*}$ is continuous. 
Let $p \in S(Z)$. Then, by the
continuity of $V(x)^{*}$ for arbitrary $x \in X$, for some $q \in S(Z)$ 
and $c_p(x) \geq 0$, we have
\begin{align*}
 p ( [ \fk_x^{*}f, \fk_x^{*}f ]_{\cH} )
& = p ( [ f(x), f(x) ]_{\cH} )
= p ( [ V(x)^{*}g, V(x)^{*}g ]_{\cH} ) \\
& \leq c_p(x)\, q ( [ g, g ]_{\cK} )
= c_p(x)\, q ( [ Ug, Ug ]_{\cR} ) \\
& = c_p(x)\, q ( [ f, f ]_{\cR} ),\quad f\in\cR,
\end{align*}
where $g \in \cK$ is the unique vector 
such that $Ug = V(\cdot)^{*}g = f$. Hence
the continuity of $\fk_x^{*}$ is proven.

(3). Assume that $(\cR;[\cdot,\cdot]_\cR)$ is an $\cH$-reproducing 
kernel VH-space on $X$, with reproducing kernel $\fk$.
 We let $\cK=\cR$ and define $V(x) \colon \cH \ra \cK$ by
\begin{equation}\label{e:vhxak}V(x)h=\fk_x h,\quad x\in X,\ h\in \cH.
\end{equation}
Then, $V(x) \in \cL^{*}(\cH, \cK)$, 
with $V(x)^{*} \colon f \ni \cK=\cR \mapsto f(x) \in \cH$ 
for all $x \in X$. From Lemma~\ref{l:mincont},
we see that, actually, $V(x) \in \cL_{c}^{*}(\cH, \cK)$ and that
$V(y)^{*}V(x) = \fk(y,x)$ for all $x,y \in X$.
Thus, $(\cK;V)$ is a VH-space linearisation of $\fk$.  
\end{proof}

Let us observe that, until now, we did not say anything about the existence of
reproducing kernel VH-spaces or, equivalently, of VH-space linearisations,
associated to a given positive semidefinite
$\cH$-kernel. This question is considered in the next subsection and
answered in Corollary~\ref{c:evhsl}, as a
consequence of Theorem~\ref{t:vhinvkolmo2}, by providing a necessary and
sufficient condition (b2). We present some cases when this boundedness condition is automatically 
fulfilled, for example, in Subsection~\ref{ss:cb2t} for a class of positive semidefinite kernels
having a certain property of m-topologisability, or the case when $\cH$ is a Hilbert module over a
locally $C^*$-algebra, see in Subsection~\ref{ss:kvaola}.

%%%%%%%
\subsection{Dilation in VH-Spaces}\label{ss:dvhs}
Let $\cH$ be a VH-space over an admissible space $Z$, let $\fk\colon X\times
X\ra\cL_c^*(\cH)$ be a kernel on some nonempty set $X$, and let $\Gamma$ be a
$*$-semigroup that acts at left on $X$. As in the case of VE-space operator
valued kernels, we call $\fk$ \emph{$\Gamma$-invariant} if
\begin{equation}\fk(\xi\cdot x,y)=\fk(x,\xi^*\cdot y),\quad
  \xi\in\Gamma,\ x,y\in X.
\end{equation} 
A triple $(\cK;\pi;V)$ is called
a \emph{$\Gamma$-invariant VH-space linearisation} for $\fk$ if
\begin{itemize}
\item[(ihl1)] $(\cK;V)$ is a VH-space linearisation of $\fk$.
\item[(ihl2)] $\pi\colon \Gamma\ra\cL^*_c(\cK)$ is a $*$-representation.
\item[(ihl3)] $V(\xi\cdot x)=\pi(\xi)V(x)$ for all $\xi\in\Gamma$ and all $x\in X$.
\end{itemize}
Also, $(\cK;\pi;V)$ is \emph{minimal} if the VH-space linearisation $(\cK;V)$ is
minimal, that is, $\cK$ is the closure of the linear span of $V(X)\cH$.

\begin{remark}\label{r:mininv} Let $(\cK;\pi;V)$ be a $\Gamma$-invariant
  VH-space linearisation for the positive semidefinite
  kernel $\fk\colon X\times X\ra\cL_c^*(\cH)$ and consider the minimal VH-space
  linearisation $(\cK_0;V_0)$
  as in Proposition~\ref{p:vhlvsrk}, that is, $\cK_0$ is the closure of the
  linear span of $V(X)\cH$ and $V_0\colon X\ra\cL_c^*(\cH,\cK_0)$ is defined
  by $V_0(x)h=V(x)$, for all $x\in X$ and all $h\in\cH$. We observe that for
  every $\xi\in \Gamma$, the operator $\pi(\xi)$ leaves $\cK_0$ invariant: for
  any $x\in X$ and any $h\in\cH$, by (ihl3) we have $\pi(\xi)V(x)h=V(\xi\cdot
  x)h\in \cK_0$, and then use linearity and continuity of $\pi(\xi)$. 
Thus, we can define $\pi_0\colon \Gamma\ra\cL_c^*(\cK_0)$ by
$\pi_0(\xi)k=\pi(\xi)k\in\cK_0$ for any $\xi\in\Gamma$ and any
  $k\in\cK_0$. Then, it is easy to see that $(\cK_0;\pi_0;V_0)$ is a minimal
  $\Gamma$-invariant linearisation of $\fk$. 
\end{remark}

The following is a topological version of Theorem \ref{t:vhinvkolmo}. 

\begin{theorem} \label{t:vhinvkolmo2} 
Let $\Gamma$ be a $*$-semigroup that acts 
on the nonempty set $X$ and let $\fk\colon X\times X\ra\cL_{c}^*(\cH)$ 
be a kernel, 
for some VH-space $\cH$ over an admissible space $Z$. 
Then the following assertions are equivalent:

\begin{itemize}
\item[(1)] $\fk$ is positive semidefinite, in the sense of \eqref{e:npos}, 
and invariant under the action of $\Gamma$ on $X$, that is, 
\eqref{e:invariant} holds, and, in addition, the following conditions hold:
\begin{itemize}
\item[(b1)] For any $\xi \in \Gamma$ and 
any seminorm $p \in S(Z)$, there exists a seminorm 
$q \in S(Z)$ and 
a constant $c_p(\xi)\geq 0$ such that for all $n \in \NN$,
$\{h_i \}_{i=1}^{n} \in \cH$, $\{ x_i \}_{i=1}^{n} \in X$ we have
\begin{equation*}
p ( \sum_{i,j=1}^{n} [ \fk(\xi\cdot x_i, \xi\cdot x_j)h_j, h_i ]_{\cH} )
\leq c_p(\xi)\,  q( \sum_{i,j=1}^{n} 
[ \fk(x_i, x_j)h_j, h_i ]_{\cH} ). 
\end{equation*} 
\item[(b2)] For any $x \in X$ and any seminorm $p \in S(Z)$, 
 there exists 
a seminorm 
$q \in S(Z)$ and 
a constant $c_p(x)\geq 0$ such that for all 
$n \in \NN$, $\{ y_i \}_{i=1}^{n} \in X$, 
$\{ h_i \}_{i=1}^{n} \in \cH$ we have
\begin{equation*}
p ( \sum_{i,j=1}^{n} [ \fk(x,y_i)h_i, \fk(x, y_j)h_j ]_{\cH} )
\leq c_p(x)\,  q ( \sum_{i,j=1}^{n} 
[ \fk(y_j, y_i)h_i, h_j ]_{\cH} ). 
\end{equation*}
\end{itemize}
\item[(2)] $\fk$ has a $\Gamma$-invariant VH-space
linearisation $(\cK;\pi;V)$.
\item[(3)] $\fk$ admits an $\cH$-reproducing kernel VH-space $\cR$ and 
there exists a $*$-representation $\rho\colon \Gamma\ra\cL_{c}^*(\cR)$ such that 
$\rho(\xi)\fk_xh=\fk_{\xi\cdot x}h$ for all $\xi\in\Gamma$, $x\in X$, $h\in\cH$.
\end{itemize}

In addition, in case any of the assertions \emph{(1)}, \emph{(2)}, or
\emph{(3)} holds,  
then a minimal\, $\Gamma$-invariant VH-space linearisation of $\fk$ can be 
constructed, any minimal $\Gamma$-invariant VH-space linearisation of $\fk$ 
is unique up to unitary equivalence, 
and the pair $(\cR;\rho)$ as in assertion \emph{(3)} is uniquely determined by 
$\fk$ as well.
\end{theorem}

\begin{proof}
(1)$\Ra$(2).
We consider the notation and the minimal $\Gamma$-invariant VE-space
linearisation $(\cG;V;\pi)$ defined as in \eqref{e:fezero}--\eqref{e:dop}.
Consider the VE-space $( \cG;[\cdot,\cdot]_{\cG} )$ 
with its natural topology defined as in Subsection \ref{ss:vhs}. 
We show that, for all $\xi\in\Gamma$, 
$\pi(\xi)$ is continuous
as a linear operator on the locally convex space $\cG$.
By the boundedness condition (b1), for any 
$p \in S(Z)$ there exists $q\in S(Z)$ and $c_p(\xi)\geq 0$ such that, for all
$f\in\cG$, we have
\begin{align*} p ([\pi(\xi)f,\pi(\xi)f]_{\cG} )
 & = p ([\pi(\xi^*)\pi(\xi)f,f]_{\cG} ) = p ([\pi(\xi^*
\xi)f,f]_{\cG} ) \\
& = p ( \sum_{x,y\in X} [ 
\fk(\xi^*\xi\cdot y,x)g(x), g(y) ]_{\cH} ) \\
& = p (\sum_{x,y\in X} [ 
\fk(\xi\cdot y,\xi\cdot x)g(x), g(y) ]_{\cH}) \\
& \leq c_p(\xi)\, q ( \sum_{x,y\in X} 
[\fk(y,x)g(x), g(y) ]_{\cH} ) \\
& = c_p(\xi)\, q ([f,f]_{\cG} ),
\end{align*}
where $f=Kg$ for some $g \in \cF_{0}$. 
Hence the continuity of $\pi(\xi)$ is proven.  

Let $\cK$ be the VH-space completion of the VE-space 
$\cG$. It follows that $\pi(\xi)$ extends uniquely to 
a continuous operator on $\cK$ and that $\pi$ is a $*$-representation of
$\Gamma$ in $\cL_{c}^*(\cK)$.

We now show that all the operators $V(x)$ defined as in \eqref{e:defvex}
are continuous as linear operators defined on $\cH$ and with values in $\cG$. 
Fix $x\in X$ and $p \in S(Z)$, but arbitrary. 
By Lemma \ref{l:redunt}, for some $q\in S(Z)$ and $c_p(x)\geq 0$, for all
$h\in\cH$  we have
\begin{align*}
 p ( [ V(x)h, V(x)h ]_{\cG} ) & =
p ( [ V(x)^{*}V(x)h, h ]_{\cH} ) \\
&=p ( [ \fk(x,x)h, h ]_{\cH} ) 
\leq c_p(x)\,q ( [ h, h ]_{\cH} ). 
\end{align*}
This proves the continuity of $V(x)$. 

On the other hand, the operators $V(x)^{*}$ 
obtained as in Theorem~\ref{t:vhinvkolmo} are 
continuous on $\cG$ for all $x \in X$. To see this,  using 
the boundedness condition (b2), for any 
$p \in S(Z)$ there exist $q\in S(Z)$ and $c_p(x)\geq 0$ such that, for all
$f\in \cG$ we have 
\begin{align*}
 p ( [ V(x)^{*}f, V(x)^{*}f ]_{\cH} ) & =
p ( [ f(x), f(x) ]_{\cH} )\\
& = p ( \sum_{y,z \in X}[ \fk(x,y)g(y), \fk(x,z)g(z) ]_{\cH} ) \\
& \leq c_p(x) q ( \sum_{y,z \in X} 
[ \fk(z,y)g(y), g(z) ]_{\cH} ) 
= c_p(x)\, q ( 
[ f, f ]_{\cG} ), 
\end{align*}
where $Kg=f$ for some $g\in\cF_0$. Hence
$V(x)^*$ is continuous and,
consequently, it extends uniquely to 
a continuous operator $V(x)^{*}\colon \cK \ra \cH$. 
A continuity argument establishes the fact 
that $V(x)\colon\cH\ra\cK$ is adjointable with adjoint 
$V(x)^{*}\colon\cK\ra\cH$. Hence $V(x) \in \cL_{c}^*(\cH,\cK)$. 
By \eqref{e:vexa} we obtain $V(y)^{*}V(x)=\fk(y,x)$
for all $x,y \in X$, and by
\eqref{e:vexic} $\pi(\xi)V(x) = V(\xi \cdot x)$ for
all $\xi \in \Gamma$ and $x \in X$. 
Therefore $(\cK;\pi;V)$ is a $\Gamma$-invariant
VH-space linearisation of $\fk$. Clearly,
it is minimal. 
The uniqueness of the minimal invariant VH-space 
linearisation follows as usually.

(2)$\Ra$(1). Let $(\cK;\pi;V)$ be a $\Gamma$-invariant
VH-space linearisation of $\fk$. We already know from
Theorem \ref{t:vhinvkolmo} that
$\fk$ is positive semidefinite and that $\fk$ is
invariant under the action of $\Gamma$
on $X$. To show that (b1) holds, letting $p \in S(Z)$ 
be a seminorm and $\xi\in\Gamma$,  since the operator
$\pi(\xi)$ is continuous, there exist
$q\in S(Z)$ and $c_p(\xi)\geq 0$, such that,
for all
$n \in \NN$,
$\{h_i \}_{i=1}^{n} \in \cH$, $\{ x_i \}_{i=1}^{n} \in X$, we have
\begin{align*}
p ( \sum_{i,j=1}^{n} [ \fk(\xi\cdot x_i, \xi\cdot x_j)h_j,
 h_i ]_{\cH} )
&=p ( \sum_{i,j=1}^{n} [ V(\xi\cdot x_i)^{*}
V(\xi\cdot x_j)h_j, h_i ]_{\cH} ) \\
&=p ( \sum_{i,j=1}^{n} [
V(\xi\cdot x_j)h_j, V(\xi\cdot x_i)h_i ]_{\cK} ) \\
&=p ( [
\pi(\xi)(\sum_{j=1}^{n}V(x_j)h_j ), 
\pi(\xi)(\sum_{i=1}^{n}V(x_i)h_i ) ]_{\cK} ) \\
& \leq c_p(\xi)\, q ( [ \sum_{j=1}^{n}V(x_j)h_j,
\sum_{i=1}^{n}V(x_i)h_i ]_{\cK} ) \\
& = c_p(\xi)\, q( \sum_{i,j=1}^{n} 
[ \fk(x_i, x_j)h_j, h_i ]_{\cH} ).
\end{align*}

We show that (b2) holds. 
Let $x \in X$ and $p \in S(Z)$ be fixed. Since the operator
$V(x)^*\in\cL(\cK,\cH)$ is continuous, for some $q\in S(Z)$ and $c_p(x)\geq 0$,
and arbitrary $n \in \NN$, $\{ y_i \}_{i=1}^{n} \in X$,
$\{ h_i \}_{i=1}^{n} \in \cH$, we have
\begin{align*}
p ( \sum_{i,j=1}^{n} [ \fk(x,y_i)h_i, 
\fk(x, y_j)h_j ]_{\cH} )
&= p ( [ V(x)^{*}(\sum_{i=1}^{n}V(y_i)h_i ), 
V(x)^{*} ( \sum_{j=1}^{n}V(y_j)h_j ) ]_{\cH} ) \\
& \leq c_p(x)\, q ( [ \sum_{i=1}^{n}V(y_i)h_i,
\sum_{j=1}^{n}V(y_j)h_j ]_{\cH} ) \\
&= c_p(x)\, q ( \sum_{i,j=1}^{n} 
[ \fk(y_j, y_i)h_i, h_j ]_{\cH} ). 
\end{align*}

(2)$\Ra$(3). Basically, this is a consequence of Proposition~\ref{p:lvsrk}.
Here are the details.
Let $(\cK;\pi;V)$ be a minimal $\Gamma$-invariant VH-space linearisation
of the kernel $\fk$ and the action of $\Gamma$ on $X$.
Defining $\cR$ as in \eqref{e:rehdef}, from Proposition~\ref{p:lvsrk} it 
follows that $\cR$ has a natural structure of minimal $\cH$-reproducing
kernel VH-space with reproducing kernel $\fk$.
Letting $\rho(\xi)=U\pi(\xi)U^{-1}$, 
where $U\colon\cK\ra\cR$ is 
the unitary operator defined as in \eqref{e:defu}, we obtain a 
$*$-representation of 
$\Gamma$ on the VH-space $\cR$ such that 
$\fk_{\xi\cdot x}=\rho(\xi)\fk_x$ for all 
$\xi\in\Gamma$ and $x\in X$. By
continuity of $\pi(\xi)$ for any $\xi \in \Gamma$, 
$\rho(\xi)$ is continuous for any $\xi \in \Gamma$
as well.

(3)$\Ra$(2). Assume that $(\cR;[\cdot,\cdot]_\cR)$ is an $\cH$-reproducing 
kernel VH-space on $X$, with reproducing kernel $\fk$
and $\rho\colon\Gamma\ra\cL_{c}^*(\cR)$ is a $*$-representation 
such that 
$\fk_{\xi\cdot x}=\rho(\xi)\fk_x$ for all $\xi\in\Gamma$ and $x\in X$.
 We let $\cK=\cR$ and define $V(x) \colon \cH \ra \cK$ by
\begin{equation*}V(x)h=\fk_x h,\quad x\in X,\ h\in \cH.
\end{equation*}
By Proposition~\ref{p:lvsrk}, it follows that 
$(\cK;V)$ is a VH-space linearisation of $\fk$. 
Then, letting $\pi=\rho$, $(\cK;\pi;V)$ is a minimal $\Gamma$-invariant 
VH-space linearisation of $\fk$.
\end{proof}

\begin{remarks}\label{r:czero}% C zero
(1) With notation as in Theorem~\ref{t:vhinvkolmo2}, let $\cC_0$ be a 
family of open, absolutely convex and solid neighbourhoods 
of the origin defining the topology of $Z$, 
and let $S_{\cC_0}(Z)=\{p_C\mid C\in\cC_0\}$ be defined as in 
Section~\ref{ss:as}. Then, the boundedness conditions (b1) and (b2) in the
assertion (1) of 
Theorem~\ref{t:vhinvkolmo2} can, equivalently, be stated only for all $p\in
S_{\cC_0}(Z)$. 

(2) In the particular case when $\Gamma$ is a group and $\xi^*=\xi^{-1}$ for
all $\xi\in\Gamma$, the  boundedness condition (i) in assertion (1) is 
always fulfilled, due to the $\Gamma$-invariance of the kernel $\fk$. 
\end{remarks}

As a consequence of Theorem~\ref{t:vhinvkolmo2}, for a given positive
semidefinite $\cH$-kernel $\fk$, we can show that the boundedness condition
(b2) is necessary and sufficient for the existence of a VH-space linearisation
and, equivalently, for the existence of an $\cH$-reproducing kernel VH-space
associated to $\fk$.

\begin{corollary}\label{c:evhsl}% existence of VH-space linearisations
Let $\fk$ be a positive semidefinite $\cH$-kernel on $X$, for some VH-space
$\cH$ over an admissible space $Z$. Then, the following assertions are
equivalent:
\begin{itemize}
\item[(1)] The following condition holds:
\begin{itemize}
\item[(b2)] For any $x \in X$ and any seminorm $p \in S(Z)$, there exists 
a seminorm $q \in S(Z)$ and 
a constant $c_p(x)\geq 0$ such that for all 
$n \in \NN$, $\{ y_i \}_{i=1}^{n} \in X$, 
$\{ h_i \}_{i=1}^{n} \in \cH$ we have
\begin{equation*}
p ( \sum_{i,j=1}^{n} [ \fk(x,y_i)h_i, \fk(x, y_j)h_j ]_{\cH} )
\leq c_p(x)\,  q ( \sum_{i,j=1}^{n} 
[ \fk(y_j, y_i)h_i, h_j ]_{\cH} ). 
\end{equation*}
\end{itemize}
\item[(2)] $\fk$ has a VH-space
linearisation $(\cK;V)$.
\item[(3)] $\fk$ admits an $\cH$-reproducing kernel VH-space $\cR$.
\end{itemize}
\end{corollary}

\subsection{Condition (b2) of Theorem \ref{t:vhinvkolmo2}.}\label{ss:cb2t}
Condition (b2) of Theorem \ref{t:vhinvkolmo2} for a positive semidefinite 
kernel can be considered as a weaker version of an 
inequality for positive semidefinite kernels taking values 
in $\cB^*(\cH)$, obtained in Proposition 3.2. of \cite{Gheondea}, 
where $\cH$ is a VH-space and $\cB^*(\cH)$ is the $C^*$-algebra of all
adjointable and bounded, in Loynes sense, operators 
on $\cH$, cf.\ \cite{Loynes1}. Consequently, it is natural to ask to which extent of generality 
condition (b2) is automatically satisfied or not.
Here, we show a rather general class of
$\cL_{c}^*(\cH)$ valued kernels, for an arbitrary topological VE-space $\cH$,
that guarantees the validity of condition (b2). 

We first prove a Schwarz type inequality for positive operators. This
inequality should be compared with that from Corollary~\ref{c:4schwarz} and
observed that here, in this particular case, 
the constant $4$ is improved to $1$.

\begin{lemma}\label{l:schwarzforposop}%schwarz ineq. for positive operators
Let $T \in \cL^*(\cH)$ be a positive operator on a topological VE-space $\cH$ 
over the topologically ordered $*$-space $Z$. Let $p\in S(Z)$. Then 
\begin{equation*}
p([Th,h] )\leq p([Th,Th] )^{\frac{1}{2}}
p([h,h] )^{\frac{1}{2}},\quad
h\in\cH.
\end{equation*} 
\end{lemma}

\begin{proof}
For any $h_1,h_2 \in \cH$ and  any number $\lambda >0$ we have 
\begin{align*}
0 & \leq [h_1-\lambda h_2,h_1-\lambda h_2]  \\
& =[h_1,h_1]+\lambda^2
[h_2,h_2]-\lambda[h_2,h_1] -\lambda[h_1,h_2] . 
\end{align*} 
By definition of the partial ordering on $Z$ and dividing by $\lambda$ we
obtain  
\begin{equation}\label{e:homogineq}%homogenizing inequality
[h_2,h_1] +[h_1,h_2]  \leq
\frac{1}{\lambda}[h_1,h_1] +\lambda[h_2,h_2] . 
\end{equation}
Letting $h_1:=h$, $h_2:=Th$ in \eqref{e:homogineq}, since $T$ is positive,
applying $p$ on both sides of \eqref{e:homogineq}, and taking into account
that $p$ is increasing, we obtain 
\begin{equation}\label{e:homogineq2}
2p([Th,h] ) \leq \frac{1}{\lambda}p([h,h] )+\lambda
p([Th,Th] ),\quad h\in\cH.
\end{equation}  
Since the left side of \eqref{e:homogineq2} does not depend on $\lambda$
it follows that
\begin{equation*}p([Th,h])\leq \frac{1}{2}\inf_{\lambda>0} \bigl(
\frac{1}{\lambda}p([h,h] )+\lambda p([Th,Th])\bigr)=p([Th,Th] )^{\frac{1}{2}}
p([h,h] )^{\frac{1}{2}},
\end{equation*} which is the required inequality.
\end{proof}

We now reformulate Lemma \ref{l:schwarzforposop} in case of a
positive semidefinite kernel. 

\begin{lemma}\label{l:schwarzforposker}
Let $\fk\colon X\times X\ra\cL^*(\cH)$ be a positive semidefinite kernel, 
where $\cH$ is a 
topological VE-space  over a topologically ordered $*$-space $Z$. Then 
for every $p\in S(Z)$, $n\in\NN$, $x, \left \{ y_i \right\}_{i=1}^{n}\in X$ 
and $\left \{ h_i \right\}_{i=1}^{n}\in\cH$ we have the inequality  
\begin{align*}
&p(\sum_{i,j=1}^n[\fk(x,y_i)h_i,\fk(x,y_j)h_j] ) \\ &\phantom{abcd}\leq 
p(\sum_{i,j=1}^n[\fk(x,x)\fk(x,y_i)h_i,\fk(x,y_j)h_j] )^{\frac{1}{2}} \
p(\sum_{i,j=1}^n[\fk(y_i,y_j)h_j,h_i] )^{\frac{1}{2}}.
\end{align*}  
\end{lemma}

\begin{proof}
By Theorem \ref{t:vhinvkolmo}, there is a minimal VE-space linearisation 
$(\cK;\pi;V)$ of $\fk$, where $\Gamma=\{\epsilon\}$ 
is the trivial $*$-group and the unital action 
of $\Gamma$ on $X$. For each fixed $x\in X$, 
consider the positive operator $T:=V(x)V(x)^*\colon \cK\ra\cK$ 
and for arbitrary $\left \{ y_i \right\}_{i=1}^{n}\in X$ and 
all $\left \{ h_i \right\}_{i=1}^{n}\in \cH$ 
the corresponding element $h:=\sum_{i=1}^n V(y_i)h_i \in\cK$. Given any 
$p\in S(Z)$, by applying Lemma~\ref{l:schwarzforposop} for these
$T$ and $h$, and 
taking into account that $V(z)^*V(t)=\fk(z,t)$ for any 
$z,t\in X$, we obtain the required inequality.   
\end{proof}

For a topological VE-space $\cH$ over the topologically 
ordered $*$-space $Z$, following \cite{Zelazko}, an operator $T\in \cL(\cH)$ 
is called \emph{m-topologisable} if for every $p\in S(Z)$, there exists 
a constant $D_p\geq 0$ and a continuous seminorm $r$ on $\cH$ 
such that, for every $n \in \NN$ and every $h\in \cH$,
\begin{equation}\label{e:mtop}
\tl{p}(T^n h)=p([T^nh,T^nh])^{\frac{1}{2}} \leq D_p^n\, r(h).
\end{equation}
Observe that m-topologisable operators are those continuous linear operators $T\colon \cH\ra\cH$ 
for which there is a certain control of the growth of their powers uniformly on $\cH$.

The inequality in the following lemma, which can be viewed as a stronger version of 
the inequality from 
Lemma~\ref{l:redunt} for the special case of an m-topologisable positive operator,
is a generalisation of the celebrated
Krein-Reid-Lax-Dieudonn\'e inequality; the iteration method through which we obtain it
is following a similar idea as that in \cite{Krein}, \cite{Reid}, 
\cite{Lax}, and \cite{Dieudonne}.

\begin{lemma}\label{l:schwarzformposop}%schwarz type ineq. for m-topologizable positive opr.
Let $T \in \cL^*(\cH)$ be an m-topologisable positive operator on a topological VE-space $\cH$ 
over the topologically ordered $*$-space $Z$. Let $p\in S(Z)$. Then there 
is a constant $C$, depending only on $T$ and $p$, such that 
\begin{equation*}
p([Th,h] )\leq C\, p([h,h] ),\quad h\in \cH.
\end{equation*}
\end{lemma} 

\begin{proof}
Using the fact $[T^{2^n}h,T^{2^n}h] =[T^{2\,2^n}h,h] $ 
for any $h\in\cH$ and any $n\in\mathbb{N}$, 
as well as successively applying Lemma \ref{l:schwarzforposop}, we obtain 
\begin{align*}
p([Th,h] ) & \leq p([Th,Th] )^{\frac{1}{2}} p([h,h] )^{\frac{1}{2}} \\
& \leq p([T^2h,T^2h] )^{\frac{1}{4}} p([h,h] )^{\frac{1}{2}+\frac{1}{4}} \\
& \,\,\, \vdots \\
& \leq p([T^{2^n}h,T^{2^n}h] )^{\frac{1}{2^{n+1}}} 
p([h,h] )^{\frac{1}{2}+\frac{1}{4}+\cdots +\frac{1}{2^{n+1}}} \\
& \leq D_p^{\frac{2^n}{2^{n}}} r(h)^{\frac{1}{2^{n}}} 
p([h,h] )^{\frac{1}{2}+\frac{1}{4}+\cdots +\frac{1}{2^{n+1}}},
\end{align*}
where the last inequality follows from the m-topologisability of $T$, 
with some constant $D_p$. 
Taking limits as $n\to\infty$, we obtain the required 
inequality with $C=D_p$. 
\end{proof}

\begin{remark}\label{r:mtopgen}
The conclusion of Lemma~\ref{l:schwarzformposop} can be obtained for a class of positive 
operators $T$ larger than that of m-topologisable ones, namely, in
\eqref{e:mtop} it is sufficient that $r\colon \cH\ra [0,+\infty)$ is an arbitrary function.
\end{remark}

It now follows that if an
m-topologisability condition is imposed on the kernel $\fk$, a stronger inequality
than that in condition (b2) of Theorem \ref{t:vhinvkolmo2} is obtained. In
particular, this kind 
of kernels always have VH-space linearisation, equivalently, their reproducing
kernel VH-spaces always exist.

\begin{proposition}\label{p:autob2mker}%automatic b2 for m topologizable kernels
Let $\fk\colon X\times X\ra\cL^*(\cH)$ be a positive semidefinite kernel, 
where $\cH$ is a 
topological VE-space  over a topologically ordered $*$-space $Z$. 
Assume that for every $x\in X$, the operator $\fk(x,x)$ is 
m-topologisable. Then, for any $x \in X$ and any seminorm $p \in S(Z)$, 
there exists a constant $c_p(x)\geq 0$ such that for all 
$n \in \NN$, $\{ y_i \}_{i=1}^{n} \in X$, 
$\{ h_i \}_{i=1}^{n} \in \cH$ we have
\begin{equation*}
p ( \sum_{i,j=1}^{n} [ \fk(x,y_i)h_i, \fk(x, y_j)h_j ]  )
\leq c_p(x)\,  p ( \sum_{i,j=1}^{n} 
[ \fk(y_j, y_i)h_i, h_j ]  ). 
\end{equation*} 
\end{proposition} 

\begin{proof}
Since $\fk(x,x)$ is an m-topologisable positive operator, by taking 
$T:=\fk(x,x)$ and $h:=\sum_{i=1}^n  \fk(x,y_i)h_i$ 
in Lemma \ref{l:schwarzformposop}, for some constant $c_p(x)\geq 0$ we have 
\begin{equation}\label{e:ineqdiagker}%inequality for the diagonal of the kernel
p(\sum_{i,j=1}^n[\fk(x,x)\fk(x,y_i)h_i,\fk(x,y_j)h_j] )
\leq c_p(x)\, p ( \sum_{i,j=1}^{n} [ \fk(x,y_i)h_i, \fk(x, y_j)h_j ]).
\end{equation} 
Then, by Lemma~\ref{l:schwarzforposker}, we have
\begin{align}\nonumber
p(\sum_{i,j=1}^n[\fk(x,y_i)h_i,\fk(x,y_j)h_j] ) & \leq 
p(\sum_{i,j=1}^n[\fk(x,x)\fk(x,y_i)h_i,\fk(x,y_j)h_j] )^{\frac{1}{2}} \,
p(\sum_{i,j=1}^n[\fk(y_i,y_j)h_j,h_i] )^{\frac{1}{2}}\\
\intertext{whence, by \eqref{e:ineqdiagker},}
 & \leq c_p(x)^{1/2}\, p( \sum_{i,j=1}^{n} [ \fk(x,y_i)h_i, \fk(x, y_j)h_j ])^{1/2}\,
p(\sum_{i,j=1}^n[\fk(y_i,y_j)h_j,h_i] )^{\frac{1}{2}}. \label{e:trouble}
\end{align}
A standard argument implies now the required inequality.
\end{proof}

\begin{remark}\label{r:trouble}
The inequality in Proposition~\ref{p:autob2mker} is stronger than condition (b2) in 
Theorem~\ref{t:vhinvkolmo2} and one can ask whether the inequality obtained in 
Lemma~\ref{l:redunt}, which does not require any extra condition on the
positive operator $T$,  
may be used instead of Lemma~\ref{l:schwarzformposop}, in order to obtain the
validity of the  
inequality in the condition (b2), in general. Unfortunately, an inspection of
the proof of  
Proposition~\ref{p:autob2mker}, more precisely \eqref{e:trouble}, 
shows that this is not the case and, if condition (b2) has to be proven in general, 
this way does not work and probably a completely new idea is needed. 
On the other hand, we do not have a counter-example
of  positive semidefinite kernels for which condition (b2) does not hold: in
view of \cite{Bonet}, 
such a counter-example should be very pathological, if exists.
The question formulated at the beginning of this subsection remains open.
\end{remark}

The next proposition shows that under very general assumptions of positivity, an m-topologisable
diagonal of the kernel propagates an even stronger continuity property throughout the kernel.

\begin{proposition}\label{p:propagation}
Let $\fk\colon X\times X\ra \cL^*_c(\cH)$ be a $2$-positive kernel, for some topological
VE-space space $\cH$ over a topologically ordered $*$-space $Z$. 
If $\fk(x,x)$ is m-topologisable for all $x\in X$ then, for any 
$x,y\in X$ and any $p\in S(Z)$, there exists $C\geq 0$ such that, for all $h\in\cH$ the following 
inequality holds
\begin{equation}\label{e:pekaxy}
p([\fk(y,x)h,\fk(y,x)h])\leq C\, p([h,h]).
\end{equation}
In particular, the linear operator $\fk(y,x)$ is m-topologisable for all $x,y\in X$.
\end{proposition}

\begin{proof} Let us fix $x,y\in X$ and $p\in S(Z)$, and let $h,g\in\cH$ vary. 
By the $2$-positivity assumption, we have
\begin{equation*} 
[\fk(x,y)g,h]+[\fk(y,x)h,g]\leq [\fk(x,x)h,h]+[\fk(y,y)g,g],
\end{equation*} 
and take $g=C_y^{-1}\fk[y,x]h$, where $C_y>0$ is a constant as 
in Lemma~\ref{l:schwarzformposop} when applied to the m-topologisable and positive operator
$T=\fk(y,y)$. 
Then, taking into account that, by Lemma~\ref{L:twopos}, $\fk(x,y)=\fk(y,x)^*$, it follows
\begin{equation*}
2C_y^{-1}[\fk(y,x)h,\fk(y,x)h]\leq [\fk(x,x)h,h]+C_y^{-2}[\fk(y,y)\fk(y,x)h,\fk(y,x)h],
\end{equation*} and, since the left side is in $Z^+$ and $p$ is increasing, we obtain
\begin{align*}
2C_y^{-1}p([\fk(y,x)h,\fk(y,x)h]) & \leq p([\fk(x,x)h,h])+C_y^{-2}p([\fk(y,y)\fk(y,x)h,\fk(y,x)h]) \\
& \leq  p([\fk(x,x)h,h])+C_y^{-1}p([\fk(y,x)h,\fk(y,x)h])\\
& \leq C_x p([h,h])+C_y^{-1}p([\fk(y,x)h,\fk(y,x)h]),
\end{align*}
which provides the required inequality \eqref{e:pekaxy}, with $C=C_x C_y$.

Finally, given arbitrary $x,y\in X$, for any $p\in S(Z)$ and any natural number $n\geq 1$, 
by iterating the inequality \eqref{e:pekaxy} $n$ times we get
\begin{equation*} p([\fk(y,x)^nh,\fk(y,x)^nh])\leq C^n p([h,h]),\quad h\in \cH,
\end{equation*} hence $\fk(y,x)$ is m-topologisable.
\end{proof}

\begin{remark} The conclusion of Proposition~\ref{p:propagation} can be obtained as a 
consequence of Proposition~\ref{p:autob2mker} if the assumption of 
$2$-positivity of the kernel is replaced by its positive 
semidefiniteness.
\end{remark}

\subsection{Kernels with Values Adjointable Operators on VH-Modules.}
\label{ss:kvaovhm}
In the following we point out an application of Theorem~\ref{t:vhinvkolmo2} to
linear maps with values adjointable operators on VH-modules over admissible
$*$-algebras. By definition, an \emph{admissible $*$-algebra} 
$\cA$ is a $*$-algebra that is, in the same time, an admissible space. A
\emph{VH-module $\cE$ over an admissible $*$-algebra $\cA$} is, by definition,
a VH-space over $\cA$, viewed as an admissible space, which is a right
$\cA$-module, as well. Given a VH-module $\cE$ over an admissible $*$-algebra $\cA$, 
an \emph{$\cH$-reproducing kernel VH-module over $\cA$} is just an
$\cE$-reproducing kernel VH-space over $\cA$, with definition as in 
Subsection~\ref{ss:vhslrk}, which is also a VH-module over $\cA$. We have
the following consequence of Theorem~\ref{t:vhinvkolmo2} and 
Proposition~\ref{p:veinvkolmomodule}.
 
 \begin{proposition}\label{p:vhinvkolmomodule} 
 Let $\Gamma$ be a $*$-semigroup that acts 
on the nonempty set $X$ and let $\fk\colon X\times X\ra\cL_c^*(\cH)$ 
be a kernel, for some VH-module $\cH$ over an admissible $*$-algebra $\cA$. 
Then, assertion \emph{(1)} in Theorem~\ref{t:vhinvkolmo2} 
is equivalent with each of the following assertions:
\begin{itemize}
\item[(2)] $\fk$ has a $\Gamma$-invariant VH-module (over $\cA$) 
linearisation $(\cK;\pi;V)$.
\item[(3)] $\fk$ admits an $\cH$-reproducing kernel VH-module $\cR$ and 
there exists a $*$-representation $\rho\colon \Gamma\ra\cL_c^*(\cR)$ such that 
$\rho(\xi)\fk_xh=\fk_{\xi\cdot x}h$ for all $\xi\in\Gamma$, $x\in X$, $h\in\cH$.
\end{itemize}

In addition, in case any of the assertions \emph{(1)}, \emph{(2)}, or
\emph{(3)} holds,  
then a minimal $\Gamma$-invariant VH-module 
linearisation can be constructed,
any minimal $\Gamma$-invariant VH-module 
linearisation is unique up to unitary equivalence, 
a pair $(\cR;\rho)$ as in assertion \emph{(3)} with $\cR$ 
minimal can be always obtained and, in this case, it is uniquely 
determined by $\fk$ as well.
\end{proposition}

If $\phi\colon \cB\ra\cL_c^*(\cH)$ is a linear map, for some $*$-algebra $\cB$ 
and some VH-module $\cH$ over an admissible $*$-algebra $\cA$, one can define
a kernel $\fk\colon \cB\times\cB\ra \cL^*_c(\cH)$ by
\begin{equation}\label{e:kerphi}% kernel phi 
\fk(a,b)=\phi(a^*b),\quad a,b\in\cB. 
\end{equation}
It is immediate to verify that, letting the $*$-semigroup $\cB$ act on itself
by multiplication, $\fk$ is $\cB$-invariant, in the sense of
\eqref{e:invariant}. Consequently, the following holds.

\begin{corollary}
\label{c:stinespringmodulevh}%VH-module version of the algebraic theorem
Let $\phi\colon \cB\ra\cL^*_c(\cH)$ be a linear map, for some $*$-algebra $\cB$ 
and some VH-module $\cH$ over an admissible $*$-algebra $\cA$. 
The following assertions are equivalent:
\begin{itemize}
\item[(1)] The map $\phi$ is positive semidefinite, in the sense that the
  kernel $\fk$ defined at \eqref{e:kerphi} is positive semidefinite, and
\begin{itemize}
\item[(b1)] For any $b \in \cB$ and any seminorm $p \in S(\cA)$, 
there exist a seminorm $q \in S(\cA)$ and 
a constant $c_p(b)\geq 0$ such that, for all $n \in \NN$,
$\{h_i \}_{i=1}^{n} \in \cH$, $\{ a_i \}_{i=1}^{n} \in \cB$, we have
\begin{equation*}
p ( \sum_{i,j=1}^{n} [ \phi(a_i^*b^* b a_j)h_j, h_i ]_{\cH} )
\leq c_p(b)\,  q( \sum_{i,j=1}^{n} 
[ \phi(a_i^* a_j)h_j, h_i ]_{\cH} ). 
\end{equation*} 
\item[(b2)] For any $b \in\cB$ and any seminorm $p \in S(\cA)$, 
there exist a seminorm 
$q \in S(\cA)$ and 
a constant $c_p(b)\geq 0$ such that, for all 
$n \in \NN$, $\{ a_i \}_{i=1}^{n} \in \cB$, 
$\{ h_i \}_{i=1}^{n} \in \cH$, we have
\begin{equation*}
p ( \sum_{i,j=1}^{n} [ \phi(b^*a_i)h_i, \phi(b^*a_j)h_j ]_{\cH} )
\leq c_p(b)\,  q ( \sum_{i,j=1}^{n} 
[ \phi(a_j^*a_i)h_i, h_j ]_{\cH} ). 
\end{equation*}
\end{itemize}
\item[(2)] There exist a VH-module $\cK$ over the admissible $*$-algebra $\cA$, 
a linear map $V\colon\cB\ra\cL_c^*(\cH,\cK)$, and a $*$-representation 
$\pi\colon \cB\ra\cL_c^*(\cK)$, such that:
\begin{itemize}
\item[(i)] $\phi(a^*b)=V(a)^*V(b)$ for all $a,b\in\cB$.
\item[(ii)] $V(ab)=\pi(a)V(b)$ for all $a,b\in\cB$.
\end{itemize}
\end{itemize}
In addition, if this happens, then the triple $(\cK;\pi;V)$ can always 
be chosen minimal, in the sense that $\cK$ is the closed linear span of the set 
$V(\cB)\cH$, and any two minimal triples as before are unique, modulo 
unitary equivalence.
\begin{itemize}
\item[(3)] There exist an $\cH$-reproducing kernel VH-module $\cR$ on 
$\cA$ and a $*$-representation $\rho\colon\cB\ra\cL_c^*(\cR)$ such that:
\begin{itemize}
\item[(i)] $\cR$ has the reproducing kernel $\cB\times\cB
\ni(a,b)\mapsto \phi(a^*b)\in\cL^*(\cH)$.
\item[(ii)] $\rho(a)\phi(\cdot b)h=\phi(\cdot ab)h$ for all 
$a,b\in\cB$ and $h\in\cH$.
\end{itemize}
\end{itemize}
In addition, the reproducing kernel VH-module $\cR$ as in (3) can always be 
constructed minimal and in this case it is uniquely determined by 
$\phi$.
\end{corollary}

In case the $*$-algebra $\cB$ is unital, Corollary~\ref{c:stinespringmodulevh} 
takes a form that is closer to a topological version of 
Kasparov's Theorem \cite{Kasparov} and its generalisation \cite{Joita}.

\begin{corollary}\label{c:stinespringmodule2vh} 
Let $\cB$ be a unital $*$-algebra and 
$\phi\colon\cA\ra\cL^*_c(\cH)$ a linear map, for 
some VH-module $\cH$ over an ordered $*$-algebra $\cA$. 
Then, assertion \emph{(1)} 
in Corollary~\ref{c:stinespringmodulevh} is equivalent with
\begin{itemize}
\item[(2)$^\prime$] 
There exist a VH-module $\cK$ over $\cA$, a $*$-representation
 $\pi\colon\cB\ra\cL_c^*(\cK)$, and $W\in\cL_c^*(\cH,\cK)$ such that 
 \begin{equation}\phi(b)=W^*\pi(b)W,\quad b\in\cB.\end{equation}
\end{itemize}
In addition, if this happens, then the triple $(\cK;\pi;W)$ can always be 
chosen minimal, in the sense that $\cK$ is the closed linear span of the set 
$\pi(\cA)W\cH$, and any two minimal triples as before are unique, 
modulo unitary equivalence.
\end{corollary}

\section{$*$-Representations on Hilbert Modules over Locally 
$C^*$-Algebras}\label{s:rephilbmodloccalg}%representations hilbert modules locally c^* algebras

In the following we specialise to the case when $\cH$ is a Hilbert module
over a locally $C^*$-algebra. After a review of preliminary material on
locally $C^*$-algebras and Hilbert modules over locally $C^*$-algebras, we
show, by an application of Proposition \ref{p:autob2mker}, 
that the boundedness condition (b2) in Theorem~\ref{t:vhinvkolmo2} is
automatic in this case. Then, as an application, we show how the Kasparov type
dilation theorem in \cite{Joita} can be proven from here in a rather direct way.

\subsection{Hilbert Modules over Locally $C^*$-Algebras}
\label{ss:hmolcsa}
A $*$-algebra $\cA$ that has a complete 
Hausdorff topology induced by a family of
\emph{$C^*$-seminorms}, that is, seminorms $p$ on $\cA$ that satisfy the
\emph{$C^*$-condition} $p(a^*a) = p(a)^2$ for all $a\in\cA$, is called a 
\emph{locally $C^*$-algebra} \cite{Inoue} 
(equivalent names are \emph{(Locally Multiplicatively Convex) $LMC^*$-algebras}
\cite{Schmudgen}, \cite{Mallios}, or \emph{$b^*$-algebra}
\cite{Allan}, \cite{Apostol}, 
or \emph{pro $C^*$-algebra} \cite{Voiculescu}), \cite{Phillips}. Note that, any
$C^*$-seminorm is \emph{submultiplicative},  
$p(ab)\leq p(a)p(b)$ for all $a,b\in\cA$, cf.\ \cite{Sebestyen},
and \emph{$*$-invariant}, $p(a^*)=p(a)$ for all $a\in\cA$. 
Denote the collection of all continuous $C^*$-seminorms 
by $S_*(\cA)$. Then $S_*(\cA)$ is a 
directed set under pointwise maximum seminorm, namely, 
given $p,q\in S_*(\cA)$, letting 
$r(a):=\mbox{max}\{p(a),q(a)\}$ for all $a\in\cA$, then 
$r$ is a continuous $C^*$-seminorm and $p,q\leq r$. 
Locally $C^*$-algebras were studied in \cite{Allan}, \cite{Apostol},
\cite{Inoue}, \cite{Schmudgen}, \cite{Phillips}, and \cite{Zhuraev}, to cite a
few.  

It follows from Corollary 2.8 in \cite{Inoue} that any locally $C^*$-algebra 
is, in particular, an admissible space, more precisely, 
a directed family of increasing 
seminorms generating the topology in axiom (a5$^\prime$) 
in Subsection \ref{ss:as} is $S_*(\cA)$.
Note that $S_*(\cA) \subset S(\cA)$ and, although 
they generate the same topology on $\cA$, 
these two sets are quite different. For instance, 
while $S(\cA)$ is a cone, $S_*(\cA)$ is not even stable 
under positive scalar multiplication. 

By $b(\cA)$ we denote the $C^*$-algebra of 
all bounded elements in $\cA$, i.e. all 
$a\in \cA$ such that 
$\Vert a \Vert_{\infty}:=\mbox{sup}\{p(a) \mid p\in S_*(\cA)\} < \infty$. 
Then $\Vert a \Vert_{\infty}$ defines a 
$C^*$-norm on $b(\cA)$. Also, 
$b(\cA)$ is dense in $\cA$, see \cite{Phillips} or \cite{Frago}.
% !!! precise reference is needed !!!

An \emph{approximate unit} of $\cA$ is an increasing net $(e_j)_{j\in\cJ}$ 
of positive elements in $\cA$ with $p(e_j)\leq 1$ for any 
$p\in S_*(\cA)$ and any $j\in\cJ$, satisfying 
$p(x-xe_j)\xrightarrow[j]{} 0$ and $p(x-e_jx)\xrightarrow[j]{} 0$ for 
all $p\in S_*(\cA)$ and all $x\in \cA$. For any locally $C^*$-algebra,
there exists an approximate unit, cf.\ \cite{Inoue}, \cite{Phillips}. 
% !!! precise reference is needed !!!
 
A \emph{pre-Hilbert module over a locally $C^*$-algebra $\cA$}, 
or a \emph{pre-Hilbert $\cA$-module} is a topological VE-module $\cH$ 
over $\cA$. Note that the topology on the pre-Hilbert $\cA$-module $\cH$ is 
given by the family of seminorms $\{\tilde p\}_{p\in S_*(\cA)}$, where
$\tilde p(h)=p([h,h])^{1/2}$ for all $p\in S_*(\cA)$ and all $h\in\cH$. 
A pre-Hilbert $\cA$-module $\cH$ is called a 
\emph{Hilbert $\cA$-module} if it is complete, e.g.\ see \cite{Phillips}.

Let $\cH$ be a pre-Hilbert $\cA$-module, let $p \in S_*(\cA)$ 
and let $x,y \in \cH$. Then a Schwarz type inequality holds, 
e.g.\ see \cite{Zhuraev}, as follows
\begin{equation}\label{e:schwarz}
p([ h,k]_{\cH})
\leq p([h,h]_{\cH})^{1/2}\, p([k,k]_{\cH})^{1/2},\quad h,k\in\cH.
\end{equation}

For a Hilbert $\cA$-module $\cH$ and $p \in S_*(\cA)$, 
denote $\cI_p^\cA:=\{a\in \cA\mid p(a)=0\}$, or simply $\cI_p$ when there will
be no danger of confusion on the ambient locally $C^*$-algebra, and 
$\tl{\cI}_p^\cH:=\{x\in \cH\mid [x,x]\in \cI_p\}$, or simply 
$\tl{\cI}_p$. Then $\cI_p$ 
is a closed $*$-ideal in $\cA$ and it is known, cf.\ \cite{Apostol},
that the quotient $\cA_p:=\cA/\cI_p$ is a $C^*$-algebra 
with $C^*$-norm $\Vert a+\cI_p \Vert_{\cA_p}:=p(a)$ for 
$a \in \cA$.
Also, $\tl{\cI}_p$ is a closed $\cA$-submodule in $\cH$ and 
the quotient module $\cH_p:=\cH/\tl{\cI}_p$ is a 
Hilbert module over the $C^*$-algebra $\cA_p$, with module action given by 
\begin{equation*}
(h+\tl{\cI}_p)(a+\cI_p):=ha+\tl{\cI}_p,\quad h\in\cH,\ a\in\cA, 
\end{equation*}
and gramian given by 
\begin{equation*}
[ h+\tl{\cI}_p, k+\tl{\cI}_p]_{\cH_p}:=
[ h,k]_{\cH}+\cI_p,\quad h,k\in\cH,\ a\in\cA.
\end{equation*}

On the other hand, when $\cH$ and $\cK$ are Hilbert modules 
over the same locally $C^*$-algebra $\cA$, 
 the space of all adjointable linear operators 
$T\colon\cH \rightarrow\cK$, denoted by $\cL^*(\cH,\cK)$,
has some additional properties, when compared to VH-spaces. 
Any operator $T \in \cL^*(\cH,\cK)$ 
is automatically a module map and continuous, cf.\ \cite{Weidner} or
Lemma~3.2 in \cite{Zhuraev},  
in particular, $T(h\cdot a)=T(h)\cdot a$ for all $h\in\cH$, $a\in\cA$ and 
$\cL^*(\cH,\cK) = \cL^{*}_{c}(\cH,\cK)$, see Subsection~\ref{ss:vhs} 
for notation. 

For fixed $p\in
S_*(\cA)$, any operator 
$T \in \cL^*(\cH,\cK)$ induces an adjointable, hence a continuous module map 
operator $T_p$ from the Hilbert $\cA_p$-module $\cH_p$ to 
the Hilbert $\cA_p$-module $\cK_p$, via 
\begin{equation}\label{e:quotopr}%quotient operator
T_p( h+\tl{\cI}_p^\cH ):=Th+\tl{\cI}_p^\cK,\quad h\in\cH, 
\end{equation}
with adjoint 
\begin{equation}\label{e:quotadjopr}%quotient adjoint operator
T^{*}_p( k+\tl{\cI}_p^\cK ):=T^{*}k+\tl{\cI}_p^\cH,\quad k\in\cK.
\end{equation}
Moreover, there is a constant $C \geq 0$ such that 
\begin{equation}\label{e:oprbddp}%operator boundedness with same p
\tl{p}_{\cK}(Th)\leq C\, \tl{p}_{\cH}(h),\quad h\in\cH,
\end{equation} e.g.\ see  Lemma~3.2 in \cite{Zhuraev}.

A topology on $\cL^*(\cH,\cK)$ can be defined via the collection of
seminorms $\{\overline p_{\cH,\cK}\}_{p\in S_*(\cA)}$: for arbitrary 
$p\in S_*(\cA)$, 
\begin{equation}\label{e:oprsn}%operator seminorm
\overline{p}_{\cH,\cK}(T):=\Vert T_p \Vert,\quad T\in\cL^*(\cH,\cK),  
\end{equation} where $\|\cdot\|$ denotes the operator norm in 
$\cL^*(\cH_p,\cK_p)$,
equivalently, $\Vert T_p \Vert$ is the infimum of all $C\geq 0$ satisfying 
inequality \eqref{e:oprbddp}. 
For the case $\cH=\cK$, these seminorms become 
$C^*$-seminorms and they turn $\cL^*(\cH)$ into a locally 
$C^*$-algebra, c.f.\  \cite{Phillips} and \cite{Zhuraev}. 

For a locally $C^*$-algebra $\cA$, let $M_n(\cA)$ 
denote the $*$-algebra of all $n \times n$ matrices 
over $\cA$. $M_n(\cA)$ becomes a locally $C^*$-algebra 
considered with the topology generated by the $C^*$-seminorms 
\begin{equation*}
p_n( [a_{ij}]_{i,j=1}^{n}):=
\Vert [a_{ij}+\cI_p]_{i,j=1}^{n} \Vert_{M_n(\cA_p)},\quad [a_{ij}]_{i,j=1}^n\in M_n(\cA),
\end{equation*}
where $\|\cdot\|_{M_n(\cA_p)}$ is the $C^*$-norm on 
the $C^*$-algebra $M_n(\cA_p)$.

\subsection{Kernels with Values Adjointable Operators in Hilbert Locally
  $C^*$-Modules.}\label{ss:kvaola}
Let $\cH$ be a Hilbert module over 
a locally $C^*$-algebra $\cA$ and $\fk\colon X\times X\ra \cL^*(\cH)$ a 
positive semidefinite kernel. Then, for each seminorm 
$p \in S_*(\cA)$, a kernel 
\begin{equation}\label{e:quotker}%definition of quotient kernel
\fk_p:X \times X \ra \cL^{*}(\cH_p),\quad  
\fk_{p}(x,y):=\fk(x,y)_p\quad \mbox{for all}\,\, x,y \in X 
\end{equation}
is defined, where $\fk(x,y)_p$ is defined as in \eqref{e:quotopr}. It is
easy to see that $\fk_p$ is positive semidefinite.

An 
\emph{$\cH$-reproducing kernel Hilbert $\cA$-module $\cR$} is a 
Hilbert $\cA$-module, which satisfies, along with (vhrk2) 
and (vhrk3),
\begin{itemize}
\item[(vhrk1)$'$] $\cR$ is a submodule of the $\cA$-module $\cF(X;\cH)$, 
with all algebraic operations.  
\end{itemize}
Note that, in this case, the axiom (vhrk4) is automatically satisfied due to
the fact that, in the case of Hilbert modules over locally $C^*$-algebras, any
adjointable, hence continuous, linear operator has continuous adjoint.

The following lemma shows that, in this special case of kernels with values
adjointable operators on Hilbert modules over locally $C^*$-algebras, the
boundedness condition (b2) in Theorem~\ref{t:vhinvkolmo2} is automatic.

\begin{lemma}\label{l:autoadjbdd}
%auto bddness of adjoint continuity condition e.g. condition (b2) 
Let $\cA$ be a locally $C^{*}$-algebra, let $\cH$ be a 
Hilbert $\cA$-module, let $X$ be a nonempty set and let 
$\fk\colon X \times X \ra \cL^{*}(\cH)$ be a positive 
semidefinite kernel. Then  
for any seminorm $p \in S_*(\cA)$ 
and any $x \in X$ there exists  
a constant $c_p(x)\geq 0$ such that for all $\{ y_i \}_{i=1}^{n} \in X$,
$\{ h_i \}_{i=1}^{n} \in \cH$ we have
\begin{equation*}
p ( \sum_{i,j=1}^{n} [ \fk(x,y_i)h_i, \fk(x, y_j)h_j ]_{\cH} )
\leq c_p(x)\, p ( \sum_{i,j=1}^{n} 
[ \fk(y_j, y_i)h_i, h_j ]_{\cH} ). 
\end{equation*}
\end{lemma}

\begin{proof}
By Proposition \ref{p:autob2mker}, it is enough to 
show that $\fk(x,x)$ is an m-topologisable operator 
for every $x\in X$. We use Lemma~3.2 in \cite{Zhuraev}, more precisely,  
we let $T:=\fk(x,x)$ in inequality \eqref{e:oprbddp} and get 
a constant $c_p(x)\geq 0$ such that 
\begin{equation*}
\tl{p}(\fk(x,x)h)\leq c_p(x)\, \tl{p}(h),\quad h\in\cH.
\end{equation*} 
This gives
\begin{equation*}
\tl{p}(\fk(x,x)^n h)\leq c_p(x)\, \tl{p}(\fk(x,x)^{n-1} h)\leq\dots
\leq c_p(x)^{n-1}\, \tl{p}(\fk(x,x)h)\leq 
c_p(x)^{n}\, \tl{p}(h)
\end{equation*} 
for all $n\in\NN$ and $h\in\cH$. Therefore, the operator 
$\fk(x,x)$ is m-topologisable.  
\end{proof}

As a consequence of the previous lemma and
Proposition~\ref{p:vhinvkolmomodule}, we have

\begin{theorem}\label{t:vhinvkolmo3} 
Let $\Gamma$ be a $*$-semigroup that acts 
on the nonempty set $X$ and let $\fk\colon X\times X\ra\cL^*(\cH)$ 
be a kernel, 
for some Hilbert module $\cH$ over a locally $C^*$-algebra $\cA$. 
Then the following assertions are equivalent:

\begin{itemize}
\item[(1)] $\fk$ is positive semidefinite 
and invariant under the action of $\Gamma$ on $X$ and, additionally, 
the following condition hold:
\begin{itemize}
\item[(b1)] For any $\xi \in \Gamma$ and 
any seminorm $p \in S(Z)$, there exists a seminorm 
$q \in S_*(A)$ and 
a constant $c_p(\xi)\geq 0$ such that for all $n \in \NN$,
$\{h_i \}_{i=1}^{n} \in \cH$, $\{ x_i \}_{i=1}^{n} \in X$ we have
\begin{equation*}
p ( \sum_{i,j=1}^{n} [ \fk(\xi\cdot x_i, \xi\cdot x_j)h_j, h_i ]_{\cH} )
\leq c_p(\xi)\,  q( \sum_{i,j=1}^{n} 
[ \fk(x_i, x_j)h_j, h_i ]_{\cH} ). 
\end{equation*} 
\end{itemize}
\item[(2)] $\fk$ has a $\Gamma$-invariant Hilbert $\cA$-module
linearisation $(\cK;\pi;V)$, that is, 
\begin{itemize}
\item[(ihl1)] $(\cK;V)$ is a Hilbert $\cA$-module linearisation of $\fk$.
\item[(ihl2)] $\pi\colon \Gamma\ra\cL^*(\cH)$ is a $*$-representation.
\item[(ihl3)] $V(\xi\cdot x)=\pi(\xi)V(x)$ for 
all $\xi\in\Gamma$ and all $x\in X$.
\end{itemize}
\item[(3)] $\fk$ admits an $\cH$-reproducing kernel Hilbert $\cA$-module 
$\cR$ and there exists a $*$-representation 
$\rho\colon \Gamma\ra\cL^*(\cR)$ such that 
$\rho(\xi)\fk_xh=\fk_{\xi\cdot x}h$ for all $\xi\in\Gamma$, $x\in X$, $h\in\cH$.
\end{itemize}
\end{theorem}

As a consequence of the previous theorem, it follows that positive semidefinite
kernels with values adjointable operators on Hilbert modules over locally
$C^*$-algebras always have Hilbert modules linearisations, equivalently, they 
admit reproducing kernel Hilbert modules.

\begin{corollary}\label{c:vhinvkolmohm}
Let $\fk\colon X\times X\ra\cL^*(\cH)$ 
be a kernel on a nonempty set $X$, 
for some Hilbert module $\cH$ over a locally $C^*$-algebra $\cA$. 
Then the following assertions are equivalent:

\begin{itemize}
\item[(1)] $\fk$ is positive semidefinite. 
\item[(2)] $\fk$ has a  Hilbert $\cA$-module
linearisation $(\cK;V)$.
\item[(3)] $\fk$ admits an $\cH$-reproducing kernel Hilbert $\cA$-module 
$\cR$.
\end{itemize}
\end{corollary}

\subsection{Completely Positive Maps.}\label{ss:cpm}
Let $\cH$ be a Hilbert $\cA$-module for some 
locally $C^*$-algebra $\cA$. Let $\cB$ 
be another locally $C^*$-algebra, let
$\phi \colon \cB \ra \cL^{*}(\cH)$ be a linear map, 
and consider the kernel $\fk$ 
associated to $\phi$ as in \eqref{e:kerphi}, that is, $\fk(a,b)=\phi(a^*b)$
for all $a,b\in \cB$. 
Then $\fk$ is invariant under the (multiplicative) 
action of $\cB$ on itself. Keeping in mind that, any $*$-algebra is, in
particular, a (multiplicative) $*$-semigroup, note that a
$\cB$-invariant Hilbert module linearisation of $\fk$ simply
is a $\cB$-invariant VH-space linearisation $(\cK;\pi;V)$ of $\fk$, 
such that $\cK$ is a Hilbert $\cA$-module.

\begin{remark}\label{r:psdcp}% positive semidefinite completely positive
Let $\cC$ and $\cD$ be locally
$C^*$-algebras. For a linear map $\phi:\cC\ra \cD$, recall that 
$\phi$ is \emph{completely positive} if
for every $n \in \NN$, the map 
$\phi^{(n)}\colon M_n(\cC)\ra M_n(\cD)$, 
$[a_{ij}]_{i,j=1}^{n} \mapsto [\phi(a_{ij})]_{i,j=1}^{n}$ 
is positive. Let $\fk$ be the kernel
associated to $\phi$ as in \eqref{e:kerphi}. Then 
$\fk$ is positive semidefinite if and only if 
$\phi$ is completely positive. This follows from the 
fact that any positive matrix in $M_n(\cC)$
can be written as the sum of positive matrices of form
$[x_{i}^{*}x_j]_{i,j=1}^{n}$, e.g.\ see \cite{Paulsen}.\end{remark}

Recall the definition of the strict topology on $\cL^*(\cH,\cK)$ for
two VH-spaces $\cH$ and $\cK$ over the same admissible space $Z$ in
Subsection~\ref{ss:as}. 
Given a linear completely positive map 
$\phi:\cB\ra\cL^*(\cH)$ for $\cB$ a 
locally $C^*$-algebra, $\cH$ a VH-space 
over a topologically admissible space $Z$, 
we say that $\phi$ is \emph{strict} if 
$(\phi(e_i))_i$ is a Cauchy net in the strict 
topology for some approximate unit $(e_i)_i$ in $\cB$.

\begin{theorem}[Theorem 4.6 in \cite{Joita}]
\label{t:hilmodkasp}%hilbert module theorem kasparov type
Let $\cA$ and $\cB$ be locally $C^*$-algebras 
and $\cH$ be a Hilbert
module over $\cA$. Let 
$\phi: \cB \ra \cL^{*}(\cH)$ be a linear map. Then the following are equivalent:
\begin{itemize}
\item[(1)] $\phi$ is a completely positive, strict, continuous 
map. 
\item[(2)] There exists $\cK$ a Hilbert
module over $\cA$, a continuous $*$-representation 
$\pi: \cB \ra \cL^{*}(\cK)$ and $W \in \cL^{*}(\cH, \cK)$
such that $\phi(a) = W^{*}\pi(a)W$ for all $a \in \cB$. 
\end{itemize}

Moreover, in case any of assertions \emph{(1)}, \emph{(2)} 
holds, the 
space $\cK$ in \emph{(2)} can be constructed minimal,  
in the sense that 
$\cK$ is the closure of $\lin\{ \pi(b)Wh \mid b \in \cB,\, h \in \cH \}$, 
and any such minimal Hilbert module is unique up to unitary equivalence. 
\end{theorem} 

We show that this theorem can be obtained as a consequence of our
Theorem~\ref{t:vhinvkolmo3} which tells us that, basically, we have to take
care of two technical obstructions: the boundedness condition (b1)
and the lack of unit of the algebra $\cB$.
We first prove two technical results that will be needed 
for solving the obstruction with the boundedness condition (b1). 

The following lemma uses an idea from the proof of Theorem 2.4 
in \cite{Murphy}.

\begin{lemma}\label{l:murphybddness}
%essentially Murphy's boundedness of $\pi(b)$ argument
Let $\cA$ be a $C^*$-algebra and $\cH$ be a Hilbert $C^*$-module 
over $\cA$. Let $\cB$ be a $C^*$-algebra and
$\phi\colon \cB \ra \cL^{*}(\cH)$ be a completely positive map. 
Then, for any $b \in \cB$ there exists 
a constant $c(b)\geq 0$ such that, for all $n \in \NN$,
$\{h_i \}_{i=1}^{n} \in \cH$, $\{ x_i \}_{i=1}^{n} \in \cB$ we have
\begin{equation*}
 \Vert \sum_{i,j=1}^{n} [ \phi(x_i^*b^* b x_j)h_j, h_i ]_{\cH} \Vert_\cA
\leq c(b) \, \Vert \sum_{i,j=1}^{n} 
[ \phi(x_i^* x_j)h_j, h_i ]_{\cH} \Vert_\cA. 
\end{equation*} 
\end{lemma}

\begin{proof}
We use Theorem~\ref{t:vhinvkolmo}, with $\Gamma=X=\cB$,
in order to obtain a minimal 
$\cB$-invariant VE-space linearisation $(\cK;\pi;V)$, 
where $\cK$ is a VE-space over $\cA$, $\pi\colon \cB \ra \cL^{*}(\cK)$ 
is a $*$-representation and $V\colon \cB \ra \cL^{*}(\cH,\cK)$. As in
Proposition~\ref{p:veinvkolmomodule}, $\cK$ is a VE-module over $\cA$ and, 
via \eqref{e:qujeh}, it is is a pre-Hilbert $\cA$-module. 

Consider $\tl{\cB}=\cB\oplus\CC$, the unitization 
of the $C^*$-algebra $\cB$. Then 
$\pi$ extends uniquely to a $*$-representation 
$\tl{\pi}:\tl{\cB}\ra\cL^{*}(\cK)$, where
$\tl{\pi}( (a,\lambda) )\colon=\pi(a)+\lambda I_\cK$. 
Let $u \in \tl{\cB}$ 
be a unitary element. It is straightforward to check that
$\tl{\pi}(u)$ is a unitary operator, hence continuous. 

Now, consider arbitrary $b \in \cB$ and 
let $u_i \in \tl{\cB}$ be unitary elements 
and $\lambda_i \in \CC$ be scalars
such that $b=\sum_{i=1}^{m}\lambda_i u_i$. Then 
\begin{equation*}
\pi(b)=\tl{\pi}(b)=
\tl{\pi}(\sum_{i=1}^{m}\lambda_i u_i)=
\sum_{i=1}^{m}\lambda_i \tl{\pi}(u_i),
\end{equation*}
therefore $\pi(b)\colon\cK\ra\cK$ is continuous. 
Taking into account that $\cK$ is topologised
by the norm $\cK\ni k\mapsto \|[k,k]_\cK\|_\cA$, this means that, there exists a
constant $c(b)\geq 0$ such that
\begin{equation*}
\|[\pi(b)k,\pi(b)k]_\cK\|\leq c(b)\|[k,k]_\cK\|_\cA,\quad k\in\cK,
\end{equation*} 
whence, in view of \eqref{e:fezero}--\eqref{e:vexic},
we obtain the required inequality. 
\end{proof}

\begin{lemma}\label{l:autobdd}%auto SzNagy bddness for hilb mod opr kernels
Let $\cA$ be a locally $C^*$-algebra and $\cH$ be a Hilbert module 
over $\cA$. Let $\cB$ be a locally $C^*$-algebra and let 
$\phi\colon \cB \ra \cL^{*}(\cH)$ be a continuous and completely positive map. 
Then, for any $b \in B$ and any $p \in S_*(\cA)$, there exists 
a constant $c_p(b)\geq 0$ such that, for all $n \in \NN$,
$\{h_i \}_{i=1}^{n} \in \cH$, $\{ x_i \}_{i=1}^{n} \in \cB$, 
we have
\begin{equation*}
p(\sum_{i,j=1}^{n} [ \phi(x_i^*b^*b x_j)h_j, h_i ]_{\cH} )
\leq c_p(b) \, p(\sum_{i,j=1}^{n} 
[ \phi(x_i^* x_j)h_j, h_i ]_{\cH} ). 
\end{equation*} 
\end{lemma}

\begin{proof}
Throughout the proof, we fix $p\in S_*(\cA)$. 
Since $S_*(\cB)$ is directed and 
$\phi$ is continuous, we can find $r \in S_*(\cB)$ and 
$d_p \geq 0$ such that 
\begin{equation}\label{e:barp} % bar p
\ol{p}(\phi(x)) \leq d_p\, r(x),\mbox{ for all }x \in \cB, \end{equation}
where the seminorm $\ol{p}$ is defined as in \eqref{e:oprsn}.
If $r(x)=0$, for some $x\in\cB$, by \eqref{e:barp} 
\begin{equation*}
\ol{p}(\phi(x)) \leq 
d_p\, r(x) = 0, 
\end{equation*}
therefore $\ol{p}(\phi(x))=0$, 
and hence $\phi(x)_p=0$ on $\cH_p$.
It follows that the map  
$\phi_p:\cB_r \ra \cL^*(\cH_p)$ defined by 
\begin{equation}\label{e:quotphi} %quotient map of \phi
\phi_p( b+\cI_r^\cB ):=\phi(b)_p,\quad b\in\cB,
\end{equation}
where $\cB_r=\cB/\cI_r^\cB$, is a well defined 
linear map. Moreover, $\phi_p$ is completely positive:
this can be checked directly by considering  
the associated kernel and proving that it is 
positive semidefinite.

Finally, applying Lemma \ref{l:murphybddness} 
for the map $\phi_p$, we get that for any $b\in\cB$, considering 
its coset $b+\cI_r^\cB \in \cB_r$, there exists 
a constant $c_p(b+\cI_r^\cB)\geq 0$ such that for all $n \in \NN$, all
$h_1,\ldots,h_n\in\cH$, and all $x_1,\ldots,x_n\in\cB$, considering their
cosets
$\{h_i+\tl{\cI}_p \}_{i=1}^{n} \in \cH_p$ and
$\{ x_i+\cI_r^\cB \}_{i=1}^{n} \in \cB_r$, we have
\begin{align*}
p(\sum_{i,j=1}^{n}& [ \phi(x_i^*b^*b x_j)h_j, 
h_i ]_{\cH} )
=\Vert \sum_{i,j=1}^{n} [ \phi_p((x_i^*b^*b x_j)+\cI_r^\cB)
(h_j+\tl{\cI}_p), 
(h_i+\tl{\cI}_p) ]_{\cH_p} \Vert_{\cA_p} \\
&=\Vert \sum_{i,j=1}^{n} [ \phi_p((x_i+\cI_r^\cB)^* 
(b+\cI_r^\cB)^*(b+\cI_r^\cB)(x_j+\cI_r^\cB))(h_j+\tl{\cI}_p), 
(h_i+\tl{\cI}_p) ]_{\cH_p} \Vert_{\cA_p} \\
&\leq c_p(b+\cI_r^\cB) \, \Vert \sum_{i,j=1}^{n} 
[ \phi_p((x_i+\cI_r^\cB)^*(x_j+\cI_r^\cB))(h_j+\tl{\cI}_p), (h_i+\tl{\cI}_p) 
]_{\cH_p} \Vert_{\cA_p} \\
&=c_p(b+\cI_r^\cB) \, p(\sum_{i,j=1}^{n} 
[ \phi(x_i^*x_j)h_j, h_i ]_{\cH} ).
\end{align*}  
Since once $p$ is fixed $r$ is also fixed, 
it is clear that we can write 
$c_p(b+\cI_r^\cB)=c_p(b)$, and the lemma is proven. 
\end{proof}

\begin{proof}[Proof of Theorem~\ref{t:hilmodkasp}]
(1)$\Ra$(2). Consider the kernel $\fk(a,b)=\phi(a^*b)$, $a,b\in\cB$.
By Theorem~\ref{t:vhinvkolmo3}, 
Lemma~\ref{l:autobdd} and Lemma~\ref{l:autoadjbdd}, we get 
a minimal $\cB$-invariant Hilbert $\cA$-module linearisation $(\cK;\pi;V)$ of
$\fk$. 

We check that $V$ is linear. For $b_1,b_2,c \in \cB$ 
and $\lambda\in\CC$ we have 
\begin{align*}
V(b_1+\lambda b_2)^*V(c)&=\phi((b_1+\lambda b_2)^*c)=
\phi(b_1^*c)+\ol{\lambda}\phi(b_2^*c) \\
&=V(b_1)^*V(c)+\ol{\lambda}V(b_2)^*V(c)=
(V(b_1)+\lambda V(b_2))^*V(c)
\end{align*}
and, by the minimality of $\cK$, it follows that 
$V(b_1+\lambda b_2)=V(b_1)+\lambda V(b_2)$.

We show that $V:\cB\ra\cL^*(\cH,\cK)$ is continuous. By the continuity of
$\phi$, for any seminorm $p\in S_{*}(\cA)$, there exist 
$r\in S_{*}(\cB)$ and $c_p \geq 0$ such that
\begin{equation}\label{e:barb}
\ol{p}_{\cH}(\phi(b^*b)) 
\leq c_p\, r(b)^2,\quad b\in\cB,
\end{equation} hence
\begin{equation*}
\ol{p}_{\cH,\cK}(V(b))^2=\Vert V(b)_{p} \Vert_{\cL(\cH_p,\cK_p)}^{2}=
\Vert V(b)_{p}^*V(b)_{p} \Vert_{\cL(\cH_p)}=
\ol{p}_{\cH}(\phi(b^*b)) 
\leq c_p\, r(b)^2,\quad b\in\cB.
\end{equation*}
 This shows that $V$ is continuous and hence the mapping 
$\cB\ni b\mapsto V(b)^*\in\cL^*(\cK,\cH)$
is also continuous, since 
$\ol{p}_{\cH,\cK}(V(b))=\ol{p}_{\cK,\cH}(V(b)^*)$ 
for all $p\in S_{*}(\cA)$. 

Now let $(e_j)_{j\in\cJ}$ be an approximate unit of 
$\cB$ with respect to which $\phi$ 
is strict. Since $V(e_j)^*V(b)=\phi(e_jb)$ 
and $e_jb\xrightarrow[j]{} b$ for any $b\in\cB$, it follows that, for any $p\in
S_*(\cA)$, we have 
\begin{equation*}
\tl{p}_{\cH}(V(e_j)^*V(b)h-\phi(b)h)=
\tl{p}_{\cH}(\phi(e_jb-b)h)\leq 
\ol{p}_{\cH}(\phi(e_jb-b))\, \tl{p}_{\cH}(h) \ra 0,\quad b\in\cB,\ h\in\cH. 
\end{equation*} 
It follows that $V(e_j)^*y$ converges to 
$\sum_{l=1}^n\phi(b_l)h_l$ whenever 
$y=\sum_{l=1}^n V(b_l)h_l$, 
i.e.\ for all $y\in\cK_0$. Let 
$p\in S_*(\cA)$. Since, $\cB\ni b\mapsto V(b)^*\in\cL^*(\cK,\cH)$ 
is continuous, there exists $r\in S_{*}(\cB)$ 
such that \begin{equation*}\label{e:vadjau}%V adjoint approximate unit
\ol{p}_{\cK,\cH}(V(e_i-e_j)^*)\leq d \, r(e_i-e_j) \leq c,\quad i,j\in\cJ,
\end{equation*} 
with $d\geq 0$ and $c > 0$ some constant numbers independent of $i,j\in\cJ$. 
Given $\epsilon > 0$, choose $y_0\in\cK_0$ such that 
\begin{equation*}
\tl{p}_{\cK}(y-y_0)\leq \frac{\epsilon}{2c}\end{equation*} 
and $i,j\in\cJ$ such that 
\begin{equation*}
\tl{p}_{\cH}(V(e_i-e_j)^*y_0)\leq\frac{\epsilon}{2}.\end{equation*}
Using these inequalities we have 
\begin{align*}
\tl{p}_{\cH}&(V(e_i)^*y-V(e_j)^*y)
=p([V(e_i)^*y-V(e_j)^*y,V(e_i)^*y-V(e_j)^*y]_{\cH})^{\frac{1}{2}} \\
&=p([V(e_i-e_j)^*y_0+V(e_i-e_j)^*(y-y_0),V(e_i-e_j)^*y_0
+V(e_i-e_j)^*(y-y_0)]_{\cH})^{\frac{1}{2}}\\ 
&=\tl{p}_{\cH}(V(e_i-e_j)^*y_0+V(e_i-e_j)^*(y-y_0)) \\
&\leq \tl{p}_{\cH}(V(e_i-e_j)^*y_0)+\tl{p}_{\cH}(V(e_i-e_j)^*(y-y_0))\\
&\leq \frac{\epsilon}{2}+
\ol{p}_{\cK,\cH}(V(e_i-e_j)^*)\, \tl{p}_{\cK}(y-y_0) \\
&\leq \frac{\epsilon}{2}+c\frac{\epsilon}{2c}
=\epsilon.
\end{align*}
Hence $(V(e_j)^*y)_{j\in\cJ}$ is a Cauchy net in $\cH$ 
for all $y\in\cK$, hence convergent. 

Now let $p\in S_{*}(\cB)$ 
 and assume that $i,j\in\cJ$ are such that $j\leq i$, hence
$0\leq e_j\leq e_i$. Since $e_i,e_j \in b(\cB)$, 
the $C^*$-algebra of all 
bounded elements of $\cB$, 
we have $(e_i-e_j)\leq(e_i-e_j)^2$, hence 
\begin{align*}
\tl{p}_{\cK}(V(e_i)h-V(e_j)h)^2&=
p([h,(V(e_i)-V(e_j))^*(V(e_i-e_j))h]_{\cH}) \\
&=p([h,\phi((e_i-e_j)^2)h]_{\cH}) \\
&\leq p([h,\phi(e_i-e_j)h]_{\cH}).
\end{align*}
Since $(\phi(e_j))_j$ is a Cauchy net for the strict topology 
of $\cL^*(\cH)$, by Lemma 
\ref{l:schwarzforposop} with $T:=\phi(e_i-e_j)$ and a standard argument, 
it follows that the net $(V(e_j)h)_j$ is Cauchy 
in $\cK$. Hence we have that 
$(V(e_j))_j$ is a Cauchy net in 
$\cL^*(\cH,\cK)$. By Lemma \ref{l:compstrict}, 
$\cL^*(\cH,\cK)$ with the strict topology 
is complete, hence there is 
$W\in\cL^*(\cH,\cK)$ such that 
$V(e_j)$ converges to $W$, with respect to the strict topology. 

We prove now that the $*$-representation $\pi\colon\cB\ra\cL^*(\cK)$ is
continuous. Let $p\in S_*(\cA)$ arbitrary. Since $\phi\colon\cB\ra\cL^*(\cH)$
is continuous, there exist $r\in S_*(\cB)$ and a constant $c_p\geq 0$ such
that \eqref{e:barp} holds.
Define $\pi_p\colon \cB_r\ra \cL^*(\cK_p)$ by
\begin{equation}\label{e:ppb}% pi p b 
\pi_p(b+\cI_r^\cB):=\pi(b)_p,\quad b\in\cB.
\end{equation}
In order for the definition in \eqref{e:ppb} to be correct, we have to show
that, if $b\in\cB$ is such that $r(b)=0$ then $\pi(b)_p=0$. Indeed, first
observe that, since $r$ is submultiplicative, from \eqref{e:barb} it follows
that, for any $x,y\in\cB$, we have
 $\overline{p}_\cH(\phi(y^*bx))=0$, that is, $\phi(y^*bx)h\in \tl{\cI_p}$ for all 
$h\in\cH$. Then,
for arbitrary $h,g\in\cH$ and $x,y\in \cB$, we have
\begin{align*}[\pi(b)_p(V(x)h+\tl{\cI_p}),(V(y)g+\tl{\cI_p})]_{\cK_p}
& =[V(y)^*\pi(b)V(x)h,g]_\cH+\cI_p\\
& =  [\phi(y^*bx)h,g]_\cH+\cI_p=\cI_p.
\end{align*} 
Since $\cK_0$, the span of $V(\cB)\cH$, is dense in $\cK$, it follows that
$\pi_p(b+\cI_r^\cB)=0$ hence, $\pi_p$ in \eqref{e:ppb} is correctly defined. It
is easy to see that $\pi_p$ is a $*$-morphism of the $C^*$-algebra
$\cB_r$ with values in the $C^*$-algebra $\cL^*(\cK_p)$, hence bounded. Letting
$d_p=\|\pi_p\|\geq 0$, where $\|\pi_p\|$ denotes 
the operator norm of this $*$-morphism $\pi_p$, it follows that
\begin{equation*}\overline{p}_\cK(\pi(b))\leq d_p\, r(b),\quad b\in\cB,
\end{equation*}
which proves the continuity of the $*$-representation $\pi$.

For any $b\in\cB$ and $h\in\cH$,
by the continuity of $V$ and of $\pi(b)$, we have 
\begin{equation*}
\pi(b)Wh=\lim_j\pi(b)V(e_j)h=
\lim_j V(be_j)h=V(b)h,
\end{equation*}
hence $\pi(b)W=V(b)$. 
Since the span of ${V(\cB)\cH}$ is dense in $\cK$, it follows 
that the span of ${\pi(\cB)W\cH}$ is dense in $\cK$. Finally, for 
any $h\in\cH$ and $b\in\cB$ we have 
\begin{equation*}
W^*\pi(b)Wh=W^*V(b)h=\lim_iV(e_i)^*V(b)h
=\lim_i\phi(e_ib)h=\phi(b)h,
\end{equation*} 
hence $W^*\pi(b)W=\phi(b)$. Uniqueness 
up to unitary equivalence follows as usually.

(2)$\Ra$(1). It can be shown, as in the proof of (2)$\Ra$(1) 
of Theorem \ref{t:vhinvkolmo}, that the associated kernel $\fk$ to $\phi$ is 
positive semidefinite hence, as in Remark \ref{r:psdcp}, we have that $\phi$ is 
completely positive. 

Since the span of
${\pi(\cB)W\cH}$ is dense in $\cK$ and $\pi$ 
is continuous, it follows that 
$\pi(e_i)\xrightarrow[i]{} I_{\cK}$ strictly 
for any approximate unit $(e_i)_i$ of $\cB$, 
where $I_{\cK}$ is the identity operator of $\cK$. 
From this we obtain that $\phi(e_i)\xrightarrow[i]{} W^*W$ strictly.

On the other hand, since $\phi(b)=W^*\pi(b)W$ 
for all $b\in\cB$, and the maps $W^*$, 
$W$, $\pi(b)$ and $\pi$ are continuous, 
it follows that $\phi$ is continuous.  
\end{proof}

\begin{remark}
During the proof of the implication (1)$\Ra$(2) from Theorem~\ref{t:hilmodkasp}, while 
proving that $(V(e_i)^*y)_i$ is a Cauchy net 
for any $y\in\cK$, one can also 
use the Schwarz inequality \eqref{e:schwarz} instead 
of subadditivity of the seminorm $\tl{p}_{\cH}$. An even 
simpler approach is to use inequality \eqref{e:pehapka} in 
Subsection \ref{ss:vhs} to get 
\begin{equation*}
p([h_1+h_2,h_1+h_2]_{\cH})\leq 2(p([h_1,h_1]_{\cH})+p([h_2,h_2]_{\cH}))
\end{equation*}
for any $h_1,\,h_2\in\cH$.
Using this inequality with 
$h_1=V(e_i-e_j)^*y_0$ and 
$h_2=V(e_i-e_j)^*(y-y_0)$ provides a valid proof as well.

Similarly, while 
proving that the net $(V(e_j)h)_j$ is Cauchy 
in $\cK$, one can use the Schwarz inequality \eqref{e:schwarz} 
instead of 
Lemma \ref{l:schwarzforposop}. The details are left to the reader.  
\end{remark}

\end{document}